\documentclass[11pt]{amsart}
\usepackage[utf8]{inputenc}
\usepackage{mathrsfs} 
\usepackage{amssymb}
\usepackage{bm}
\usepackage{graphicx}
\usepackage[centertags]{amsmath}
\usepackage{amsfonts}
\usepackage{amsthm}
\usepackage{enumerate}
\usepackage{tocvsec2}
\usepackage{xcolor}
\usepackage[margin=1in]{geometry}

\newtheorem{theorem}{Theorem}[section]
\newtheorem*{theorem*}{Theorem}

\newtheorem{corollary}[theorem]{Corollary}
\newtheorem{lemma}[theorem]{Lemma}
\newtheorem{rem}[theorem]{Remark}

\newtheorem{proposition}[theorem]{Proposition}

\newtheorem{fact}[theorem]{Fact}

\theoremstyle{definition}

\newcommand{\ee}{\varepsilon}
\newcommand{\nn}{\mathbb{N}}
\newcommand{\rr}{\mathbb{R}}
\newcommand{\ord}{\textbf{Ord}}

\begin{document}
\title[$\xi$-Asymptotically uniformly flat norms]
{Power type $\xi$-Asymptotically uniformly smooth \\ and $\xi$-asymptotically uniformly flat norms}

\begin{abstract} For each ordinal $\xi$ and each $1<p<\infty$, we offer a natural, ismorphic characterization of those spaces and operators which admit an equivalent $\xi$-$p$-asymptotically uniformly smooth norm. We also introduce the notion of $\xi$-asymptotically uniformly flat norms and provide an isomorphic characterization of those spaces and operators which admit an equivalent $\xi$-asymptotically uniformly flat norm. 

Given a compact, Hausdorff space $K$, we prove an optimal renormong theorem regarding the $\xi$-asymptotic smoothness of $C(K)$ in terms of the Cantor-Bendixson index of $K$. We also prove that for all ordinals, both the isomorphic properties and isometric properties pass from Banach spaces to their injective tensor products.  

We study the classes of $\xi$-$p$-asymptotically uniformly smooth, $\xi$-$p$-asymptotically uniformly smoothable, $\xi$-asymptotically uniformly flat, and $\xi$-asymptotically uniformly flattenable operators. We show that these classes are either a Banach ideal or a right Banach ideal when assigned an appropriate ideal norm.

\end{abstract}

\author{R.M. Causey}
\address{Department of Mathematics, Miami University, Oxford, OH 45056, USA}
\email{causeyrm@miamioh.edu}

\thanks{2010 \textit{Mathematics Subject Classification}. Primary: 46B03, 46B06; Secondary: 46B28, 47B10.}
\thanks{\textit{Key words}: Szlenk index, operator ideals, ordinal ranks.}

\maketitle

\tableofcontents

\addtocontents{toc}{\setcounter{tocdepth}{1}}

\section{Introduction} Beginning with \cite{KOS}, significant attention has been devoted to equivalent, asymptotically uniformly smooth norms on Banach spaces and operators (\cite{C4}, \cite{DKLR}, \cite{GKL}, \cite{R}).  These renorming theorems have been concerned with when an equivalent asymptotically uniformly smooth norm exists, and what isomorphic invariants of the given space or operator can determine which power types, if any, are possible for the asymptotic uniform smoothness modulus of a Banach space. These results hinge upon the Szlenk index and Szlenk power type of a Banach space.  The classes of spaces (resp. operators) which admit asymptotically uniformly smooth norms coincides with the class of spaces (resp. operators) which have Szlenk index $\omega=\omega^{0+1}$.  In \cite{C3}, \cite{CD}, and \cite{LPR},  analogous renorming results were shown for spaces (resp. operators) which have Szlenk index $\omega^{\xi+1}$ for a general $\xi$, with the $\xi=0$ case recovering the previously known results.  In a number of recent results in the non-linear theory of Banach spaces, power type asymptotically uniformly smooth and asymptotically uniformly convex Banach spaces have played an important role. Given the recent, remarkable result of Motakis and Schlumprecht \cite{MS}, which proves a transfinite version of results from \cite{BKL}, the existence of the transfinite notions of asymptotic uniform smoothness and asymptotic uniform convexity are potentially very useful in proving results akin to those in \cite{MS}. The notion of asymptotic uniform flatness was pivotal in \cite{GKL2} for the Lipschitz classification of $C(K)$ spaces isomorphic to $c_0$.  We introduce for each ordinal $\xi$ the notion of $\xi$-asymptotic uniform flatness and provide an isomorphic characterization of those Banach spaces which admit an equivalent $\xi$-asymptotically uniformly flat norm. 

We first state our main renorming theorem. All required definitions will be given in subsequent sections. 

\begin{theorem} Let $A:X\to Y$ be an operator, $\xi$ an ordinal, and $1<p<\infty$. \begin{enumerate}[(i)]\item $A$ admits an equivalent $\xi$-$p$-asymptotically uniformly smooth norm if and only if $A$ satisfies $\xi$-$\ell_p$ upper tree estimates. \item $A$ admits an equivalent $\xi$-asymptotically uniformly flat norm if and only if $A$ satisfies $\xi$-$c_0$ upper tree estimates. \end{enumerate}

\end{theorem}

We also prove an optimal renorming theorem for $C(K)$ spaces regarding $\xi$-asymptotic uniform smoothness.  

\begin{theorem} Let $K$ be a compact, Hausdorff space. Then for any ordinal $\xi$, $C(K)$ admits an equivalent $\xi$-asymptotically uniformly flat norm if and only if the Cantor-Bendixson index of $K$ is less than $\omega^{\xi+1}$.

\label{re2}
\end{theorem}

Theorem \ref{re2} is a transfinite version of a result from \cite{L}, in which the $\xi=0$ case of this theorem was shown.

We also prove that each of our asymptotic properties of operators pass to injective tensor products in the optimal way.

\begin{theorem} Let $A_0:X_0\to Y_0$, $A_1:X_1\to Y_1$ be operators and let $A_0\otimes A_1:X_0\hat{\otimes}_\ee X_1\to Y_0\hat{\otimes}_\ee Y_1$ be the induced operator between injective tensor products. Let $\xi$ be any ordinal and $1<p<\infty$.  If $A_0$, $A_1$ are $\xi$-asymptotically uniformly smooth (resp. $\xi$-$p$-asymptotically uniformly smooth, $\xi$-asymptotically uniformly flat), then so is $A_0\otimes A_1$.

\end{theorem}

Finally, we study the properties of classes of $\xi$-$p$-asymptotically uniformly smooth ($\mathfrak{G}_{\xi,p}$), $\xi$-asymptotically uniformly flat $(\mathfrak{F}_\xi)$, $\xi$-$p$-asymptotically uniformly smoothable $(\mathfrak{T}_{\xi,p})$, and $\xi$-asymptotically uniformly flattenable $(\mathfrak{T}_{\xi,\infty})$ operators. Regarding these topics, we have the following.

\begin{theorem} Fix an ordinal $\xi$ and $1<p< \infty$. \begin{enumerate}[(i)]\item There exist ideal norms $\mathfrak{t}_{\xi, p}$, $\mathfrak{t}_{\xi,\infty}$ such that $(\mathfrak{T}_{\xi, p}, \mathfrak{t}_{\xi,p})$ and $(\mathfrak{T}_{\xi,\infty}, \mathfrak{t}_{\xi,p})$ are Banach ideals. \item There exist right ideal norms $\mathfrak{g}_{\xi,p}$, $\mathfrak{f}_\xi$ such that $(\mathfrak{G}_{\xi,p}, \mathfrak{g}_{\xi,p})$ and $(\mathfrak{F}_\xi, \mathfrak{f}_\xi)$ are right Banach ideals. \end{enumerate}

\end{theorem}

The author wishes to thank G. Lancien for productive discussions during the preparation of this work.

\section{Weakly null trees and moduli}

Throughout, $\mathbb{K}$ will denote the scalar field, which is either $\rr$ or $\mathbb{C}$. By ``operator,'' we shall mean continuous, linear operator between Banach spaces. 

\subsection{Trees and moduli}

Given a set $\Lambda$, we let $\Lambda^{<\nn}$ denote the finite sequences whose members lie in $\Lambda$, including the empty sequence, $\varnothing$. Given $t\in \Lambda^{<\nn}$, we let $|t|$ denote the length of $t$. For $0\leqslant i\leqslant |t|$, we let $t|_i$ denote the initial segment of $t$ having length $i$.  If $t\neq \varnothing$, we let $t^-$ denote the maximal, proper initial segment of $t$. We let $s\smallfrown t$ denote the concatenation of $s$ with $t$. Given $s,t\in \Lambda^{<\nn}$, we let $s<t$ denote the relation that $s$ is a proper initial segment of $t$.

For us, a \emph{tree on} $\Lambda$ will be a subset of $\Lambda^{<\nn}\setminus \{\varnothing\}$ such that if $\varnothing<s\leqslant t\in T$,  then $s\in T$.  A \emph{rooted tree on} $\Lambda$ will be a subset of $\Lambda^{<\nn}$ such that if $\varnothing\leqslant s\leqslant t\in T$, then $s\in T$.  If the underlying set $\Lambda$ is understood or unimportant, we will simply refer to $T$ as a tree (or rooted tree) without specifying the set $\Lambda$.

Given a tree $T$, we let $T'=T\setminus MAX(T)$, where $MAX(T)$ denotes the $<$-maximal members of $T$. That is, $MAX(T)$ denotes those $t\in T$ which have no proper extension in $T$. Note that $T'$ is also a tree. We then define the transfinite derived trees of $T$ by $$T^0=T,$$ $$T^{\xi+1}=(T^\xi)',$$ and if $\xi$ is a limit ordinal, we let $$T^\xi=\bigcap_{\zeta<\xi}T^\zeta.$$  If there exists an ordinal $\xi$ such that $T^\xi=\varnothing$, we say $T$ is \emph{well-founded} and we let $o(T)$ denote the smallest ordinal $\xi$ such that $T^\xi=\varnothing$. If $T^\xi\neq \varnothing$ for all $\xi$, then we say $T$ is \emph{ill-founded} and we write $o(T)=\infty$. However, we will not be particularly concerned with the ill-founded case. Note that all of the definitions here for trees can be also made for rooted trees. Given two trees $S,T$, we say $\theta:S\to T$ is \emph{monotone} provided that for any $s<s'$, $s,s'\in S$, $\theta(s)<\theta(s')$.  We recall a standard fact regarding trees, which we use freely throughout. A proof of this can be found in \cite{C}. 

\begin{fact} If $S,T$ are trees and $o(S)\leqslant o(T)$, there exists a monotone, length-preserving map $\theta:S\to T$.

\end{fact}

Given a directed set $D$ and a tree $T$, we let $$T.D=\{(\zeta_i, u_i)_{i=1}^n: (\zeta_i)_{i=1}^n\in T, u_i\in D\}.$$   We treat the members of $T.D$ as sequences of pairs, so that the length of $(\zeta_i, u_i)_{i=1}^n$ is $n$. Furthermore, it is clear that $T.D$ is also a tree and $(T.D)^\xi=T^\xi.D$ for any ordinal $\xi$. Furthermore, $T.D\cup \{\varnothing\}$ is a rooted tree and $(T.D\cup \{\varnothing\})^\xi= (T\cup \{\varnothing\})^\xi.D$ for any ordinal $\xi$. Note that for a sequence $t\in T.D\cup \{\varnothing\}$, $\zeta\in \Lambda$, then $t\smallfrown (\zeta, u)\in T.D$ for some $u\in D$ if and only if $t\smallfrown (\zeta,u)\in T.D$ for every $u\in D$. 

Given a directed set $D$, a tree $T$, and a Banach space $X$, we say a collection of vectors $(x_t)_{t\in T.D}\subset X$ is \emph{weakly null} provided that for any $t\in T'.D\cup \{\varnothing\}$ and any $\zeta$ such that $\{t\smallfrown (\zeta, u): u\in D\}\subset T.D$, $(x_{t\smallfrown (\zeta, u)})_{u\in D}$ is a weakly null net in $X$.  For convenience, our primary method of witnessing weakly null trees will be to have $D$ be a weak neighborhood basis at $0$ in $X$ and to have $x_{(\zeta_i, u_i)_{i=1}^n}\in u_n$, which obviously satisfies the weakly null condition. However, for some of our later applications, it will be convenient to have the more general notion of a weakly null collection defined here.  Given a collection $(x_t)_{t\in T.D}\subset X$, a subset of this collection of the form $\{x_s: \varnothing<s\leqslant t\}$ for some $t\in T$ will be called a \emph{branch}.   We say a collection $(x_t)_{t\in T.D}$ is \emph{weakly null of order} $\xi$ if $o(T)=\xi$.

Given a directed set $D$, a tree $T$, and a Banach space $Y$, we say a collection of vectors $(y^*_t)_{t\in T.D}\subset Y^*$ is \emph{weak}$^*$ \emph{null} provided that for any $t\in T'.D\cup \{\varnothing\}$ and any $\zeta$ such that $\{t\smallfrown (\zeta, u): u\in D\}\subset T.D$, $(y^*_{t\smallfrown (\zeta, u)})_{u\in D}$ is a weak$^*$ null net. We define \emph{weak}$^*$ \emph{null} of order $\xi$ in the obvious way.

\begin{rem}\upshape We note that the definition of a weakly null collection or a weak$^*$ null collection differs from those given in \cite{CD}.  For the purposes of computing the moduli defined below, the distinctions between the two definitions of weakly null or weak$^*$ null make no difference.  In the last section of this paper, we explain why this is so. However, for the purposes of this work, our present, more restrictive definition is required. This is because in \cite{CD}, the only goal was to show, for example, that every weakly null collection of order $\omega^\xi$ admitted  a branch with a certain property. For our work we wish to show that for a weakly null collection of order $\omega^\xi$, ``most'' of the branches have that certain property. We will make this precise in the next subsection.

\end{rem}

Given  a rooted tree $T$ with order $\xi+1$, a Banach space $Y$, and $(y^*_t)_{t\in T}\subset Y^*$, we say the collection is \emph{weak}$^*$-\emph{closed of order} $\xi$ provided that for any ordinal $\zeta$ and any $t\in T^{\zeta+1}$, $$y^*_t\in \overline{\{y^*_s: s\in T^\zeta, s^-=t\}}^{\text{weak}^*}.$$    This definition of weak$^*$-closed is the same as that given in \cite{CD}.

Given an ordinal $\xi$, a non-zero operator $A:X\to Y$, and  $\sigma>0$, we define $$\varrho_\xi(\sigma, A)= \sup\Bigl\{ \inf\{\|y+Ax\|-1: t\in T.D, x\in \text{co}(x_s: \varnothing<s\leqslant t)\}\Bigr\},$$ where the supremum is taken over all $y\in B_Y$, all trees $T$ with $o(T)=\omega^\xi$, all directed sets $D$, and all weakly null collections $(x_t)_{t\in T.D}\subset \sigma B_X$. It is contained in \cite{CD} that this supremum actually need not be taken over all $T$ and $D$, and in fact we obtain the same modulus taking the supremum only over $T=\Gamma_\xi$ and $D$ to be any weak neighborhood  basis at $0$ in $X$, where $\Gamma_\xi$ is one of the special trees to be defined in the next section. However, it is convenient to state the definition of the modulus in this generality. For completeness, we define $\varrho_\xi(0, A)=0$ for any $A$ and $\varrho_\xi(\sigma, A)=0$ for all $\sigma>0$ when $A$ is the zero operator. 

For an operator $A:X\to Y$, an ordinal $\xi$, and $1<p<\infty$, we say $A$ is \begin{enumerate}[(i)]\item $\xi$-\emph{asymptotically uniformly smooth} if $\lim_{\sigma\to 0^+}\varrho_\xi(\sigma, A)/\sigma=0$, \item $\xi$-$p$-\emph{asymptotically uniformly smooth} if $\sup_{\sigma>0}\varrho_\xi(\sigma, A)/\sigma^p<\infty$, \item $\xi$-\emph{asymptotically uniformly flat} if there exists $\sigma>0$ such that $\varrho_\xi(\sigma, A)=0$. \end{enumerate} We abbreviate these properties as $\xi$-AUS, $\xi$-$p$-AUS, and $\xi$-AUF. We say a Banach space $X$ has one of these properties if its identity operator does. We note that the notions of $0$-AUS, $0$-$p$-AUS, and $0$-AUF coincide with the usual definitions of asymptotically uniformly smooth, $p$-asymptotically uniformly smooth, and asymptotically uniformly flat. We remark that renorming the domain of an operator may change specific values of the modulus $\varrho_\xi$, but each of the properties $\xi$-AUS, $\xi$-$p$-AUS, and $\xi$-AUF is invariant under renorming the domain. We say $A:X\to Y$ is $\xi$-\emph{asymptotically uniformly smoothable} (resp. $\xi$-$p$-\emph{asymptotically uniformly smoothable},  $\xi$-\emph{asymptotically uniformly flattenable}) if there exists an equivalent norm $|\cdot|$ on $Y$ such that $A:X\to (Y, |\cdot|)$ is $\xi$-AUS (resp. $\xi$-$p$-AUS, or $\xi$-AUF).

Given an ordinal $\xi$, an operator $A:X\to Y$, and $\tau>0$, we define $$\delta^{\text{weak}^*}_\xi(\tau, A)= \inf\Bigl\{\sup\|y^*+\sum_{\varnothing<s\leqslant t}y^*_s\|-1: t\in T.D\Bigr\},$$  where the supremum is taken over all $y^*\in S_{Y^*}$, all trees $T$ with $o(T)=\omega^\xi$, and all weak$^*$-null collections $(y^*_t)_{t\in T.D}\subset Y^*$ such that $\|A^*y^*_t\|\geqslant \tau$ for all $t\in T.D$. For $1\leqslant q<\infty$, an operator $A:X\to Y$, and an ordinal $\xi$, we say $A$ is \begin{enumerate}[(i)]\item \emph{weak}$^*$-$\xi$-\emph{asymptotically uniformly convex} provided $\delta^{\text{weak}^*}_\xi(\tau, A)>0$ for all $\tau>0$, \item \emph{weak}$^*$-$\xi$-$q$-\emph{asymptotically uniformly convex} provided $\inf_{0<\tau<1} \delta^{\text{weak}^*}_\xi(\tau, A)/\tau^q>0$. \end{enumerate} In the case $q=1$, this is equivalent to $\inf_{\tau>0}\delta^{\text{weak}^*}_\xi(\tau, A)/\tau>0$. We refer to these as properties of $A$ and not $A^*$, since they depend on the weak$^*$-topology on $Y^*$ coming from $Y$. We use the abbreviations weak$^*$-$\xi$-AUC,  weak$^*$-$\xi$-$q$-AUC.  We remark that these properties are invariant under renorming $X$. We will say $A$ is \emph{weak}$^*$-$\xi$-\emph{asymptotically uniformly convexifiable} (resp \emph{weak}$^*$-$\xi$-$q$-\emph{asymptotically uniformly convexifiable})  if there exists an equivalent norm $|\cdot|$ on $Y$ making $A:X\to (Y, |\cdot|)$ weak$^*$-$\xi$-AUC (resp. weak$^*$-$\xi$-$q$-AUC).

We also define some related quantities for this operator $A:X\to Y$. Given an ordinal $\xi$ and $1<p<\infty$, we let $\textbf{t}_{\xi,p}(A)$ denote the infimum of those $C>0$ such that for any $\sigma\geqslant 0$, any $y\in Y$, any tree $T$ with $o(T)=\omega^\xi$, any weakly null collection $(x_t)_{t\in T.D}\subset \sigma B_X$, $$\inf\Bigl\{\|y+Ax\|^p: t\in T.D, x\in \text{co}(x_s: \varnothing<s\leqslant t)\Bigr\} \leqslant \|y\|^p+C^p\sigma^p.$$  We observe the convention that $\textbf{t}_{\xi,p}(A)=\infty$ if no such $C$ exists.   We define $\textbf{t}_{\xi,\infty}(A)$ to be the infimum of those $C>0$ such that for any $\sigma\geqslant 0$, any $y\in Y$, any tree $T$ with $o(T)=\omega^\xi$, and any weakly null collection $(x_t)_{t\in T.D}\subset \sigma B_X$, $$\inf\Bigl\{\|y+Ax\|: t\in T.D, x\in \text{co}(x_s: \varnothing<s\leqslant t)\Bigr\} \leqslant \max \{\|y\|, C\sigma \}.$$

The following fact is an easy computation which was discussed in \cite{C4}.

\begin{proposition} Let $A:X\to Y$ be an operator and let $\xi$ be an ordinal. \begin{enumerate}[(i)]\item For $1<p<\infty$, $A$ is $\xi$-$p$-AUS if and only if $\textbf{\emph{t}}_{\xi, p}(A)<\infty$. \item $A$ is $\xi$-AUF if and only if $\textbf{\emph{t}}_{\xi, \infty}(A)<\infty$. \end{enumerate}

More precisely, for any $1<p< \infty$, there exist functions $f_p:[0, \infty)\to [0, \infty)$ and $h_p:[0, \infty)^2\to [0, \infty)$ such that for any ordinal $\xi$, any operator $A:X\to Y$, and any constants $C,C', C''$, if $\|A\|\leqslant C$ and $\sup_{\sigma>0} \varrho_\xi(\sigma, A)/\sigma^p\leqslant C'$, then $\textbf{\emph{t}}_{\xi,p}(A)\leqslant h_p(C, C')$, and if $\textbf{\emph{t}}_{\xi,p}(A)\leqslant C''$, $\sup_{\sigma>0}\varrho_\xi(\sigma, A)/\sigma^p\leqslant f_p(C'')$. 

Furthermore, if $\varrho_\xi(1/\sigma, A)=0$, then $\textbf{\emph{t}}_{\xi,\infty}(A)\leqslant \sigma +\|A\|$, and if $\textbf{\emph{t}}_{\xi, \infty}(A)<\sigma$, $\varrho_\xi(1/\sigma, A)=0$. 

\label{quant}

\end{proposition}

\begin{lemma} Fix an ordinal $\xi$, $1<p<\infty$, and an operator $A:X\to Y$. \begin{enumerate}[(i)]\item $A$ is $\xi$-AUS if and only if it is weak$^*$-$\xi$-AUC. \item $A$ is $\xi$-$p$-AUS if and only if it is weak$^*$-$\xi$-$q$-AUC, where $1/p+1/q=1$. \item $A$ is $\xi$-AUF if and only if it is weak$^*$-$\xi$-$1$-AUC. \end{enumerate}

\label{duality}
\end{lemma}

\begin{proof} Items $(i)$ and $(ii)$ were shown in \cite[Theorem $3.2$, Proposition $3.3$]{CD}. It was also shown there that if $\sigma, \tau>0$ are such that $\delta^{\text{weak}^*}_\xi(\tau, A)\geqslant \sigma \tau$, then $\varrho_\xi(\sigma, A)\leqslant \sigma \tau$.  Thus if $\inf_{\tau>0}\delta^{\text{weak}^*}_\xi(\tau, A)/\tau=\sigma>0$, $\varrho_\xi(\sigma, A) \leqslant \inf_{\tau>0} \sigma \tau=0$. This shows that if $A$ is weak$^*$-$\xi$-$1$-AUC, it is $\xi$-AUF.  

Now suppose $A$ is $\xi$-AUF and $\varrho_\xi(\sigma, A)=0$, where $\sigma>0$. Fix $\tau>0$, $y^*\in Y^*$ with $\|y^*\|=1$, and $(y^*_t)_{t\in T.D}\subset Y^*$ weak$^*$-null of order $\omega^\xi$ with $\|A^*y^*_t\|\geqslant \tau$ for all $t\in T.D$ (it was shown in \cite{CD} that if no such collection exists, $\delta^{\text{weak}^*}_\xi(\tau, A)=\infty$ for all $\tau>0$, which means $A$ is trivially weak$^*$-$\xi$-$1$-AUC).   Then, as was shown in \cite{CD}, for any $0<\delta, \theta<1$, there exist $y\in S_Y$,  a tree $S$ of order $\omega^\xi$, a weakly null collection $(x_t)_{t\in S.D}\subset B_X$, and a monotone map $d:S.D\to T.D$ such that for every $t\in S.D$, $$\text{Re\ }(y^*+\sum_{\varnothing<s\leqslant d(t)}y^*_s)(y+Ax) \geqslant 1-\delta+\theta \sigma \tau/2$$ for any $x\in \text{co}(x_s: \varnothing<s\leqslant t)$. Since $\varrho_\xi(\sigma, A)=0$, there exists $t\in S.D$ and $x\in \text{co}(x_s: \varnothing<s\leqslant t)$ such that $\|y+Ax\|\leqslant 1+\delta$, whence $$\sup_{u\in T.D} \|y^*+\sum_{\varnothing<s\leqslant u}\|\geqslant \frac{\text{Re\ }(y^*+\sum_{\varnothing<s\leqslant d(t)}y^*_s)(y+Ax) }{1+\delta}\geqslant \frac{1-\delta+\theta \sigma \tau/2}{1+\delta}.$$ Since $0<\delta, \theta<1$ were arbitrary, we deduce that $\delta_\xi^{\text{weak}^*}(\tau, A)/\tau \geqslant \sigma /2$. Since this holds for all $\tau>0$, we deduce that $A$ is weak$^*$-$\xi$-$1$-AUC.

\end{proof}

We also include here a discussion of the $\xi$-Szlenk power type of an operator and existing renorming results on this topic. Given a Banach space $X$, a weak$^*$-compact subset $K$ of $X^*$, and $\ee>0$, we let $s_\ee(K)$ denote the subset of $K$ consisting of those $x^*\in K$ such that for every weak$^*$-neighborhood $V$ of $x^*$, $\text{diam}(V\cap K)>\ee$. We define the transfinite derived sets by $$s_\ee^0(K)=K,$$ $$s^{\xi+1}(K)=s_\ee(s^\xi_\ee(K)),$$ and if $\xi$ is a limit ordinal, $$s^\xi_\ee(K)=\bigcap_{\zeta<\xi}s_\ee^\zeta(K).$$   We let $Sz(K, \ee)$ be the minimum ordinal $\xi$ such that $s_\ee^\xi(K)=\varnothing$ if such a $\xi$ exists, and otherwise we write $Sz(K, \ee)=\infty$. We let $Sz(K)=\sup_{\ee>0} Sz(K ,\ee)$, with the agreement that $Sz(K)=\infty$ if $Sz(K, \ee)=\infty$ for some $\ee>0$. Given an operator $A:X\to Y$, we let $Sz(A, \ee)=Sz(A^*B_{Y^*}, \ee)$, $Sz(A)=Sz(A^*B_{Y^*})$. If $X$ is a Banach space, we let $Sz(X, \ee)=Sz(I_X, \ee)$ and $Sz(X)=Sz(I_X)$. Then $Sz(A)$ is the \emph{Szlenk index} of $A$, and $Sz(X)$ is the Szlenk index of $X$.   

We may also define $Sz_\xi(K, \ee)$ to be the minimum ordinal $\zeta$ such that $s^{\omega^\xi \zeta}_\ee(K)=\varnothing$. If $Sz(K, \ee)\leqslant \omega^{\xi+1}$ for a weak$^*$-compact set, then $Sz_\xi(K, \ee)$ is a natural number, and we may define $$\textbf{p}_\xi(A)=\underset{\ee\to 0^+}{\lim\sup} \frac{\log Sz_\xi(A, \ee)}{|\log(\ee)|}.$$ We note that this value need not be finite.   This is the $\xi$-\emph{Szlenk power type} of $A$. For completeness, we may define $\textbf{p}_\xi(A)=\infty$ whenever $Sz(A)>\omega^{\xi+1}$. Regarding these topics, we have the following existing renorming results.

\begin{theorem}\cite[Theorem $5.3$]{CD} Let $A:X\to Y$ be an operator and let $\xi$ be an ordinal. Then $A$ is $\xi$-AUS-able if and only if $Sz(A)\leqslant \omega^{\xi+1}$.  

\label{g}

\end{theorem}

It was shown in \cite{CAlt} that $A:X\to Y$ has Szenk index not exceeding $\omega^\xi$ if and only if for every weakly null collection $(x_t)_{T.D}\subset B_X$ of order $\omega^\xi$ and every $\ee>0$, $$\inf\{\|Ax\|: t\in T.D, x\in \text{co}(x_s: \varnothing<s\leqslant t)\}=0.$$ This is equivalent to $\varrho_\xi(\sigma, A)=0$ for all $\sigma>0$, so operators with Szlenk index not exceeding $\omega^\xi$ are trivially $\xi$-AUF.

\begin{theorem}\cite[Theorem $1.2$, Theorem $2.3$]{C3} Let $A:X\to Y$ be an operator and let $\xi$ be an ordinal. \begin{enumerate}[(i)]\item If $Sz(A)\leqslant \omega^\xi$, then $A$ is asymptotically uniformly flat. \item If $Sz(A)=\omega^{\xi+1}$, then $\textbf{\emph{p}}_\xi(A)\in [1, \infty]$, and if $1/p+1/\textbf{\emph{p}}_\xi(A)=1$, then $p$ is the supremum of those $q$ such that $A$ admits an equivalent $\xi$-$q$-AUS norm. 
\end{enumerate}

\label{Szlen}

\end{theorem}

One of the main purposes of our general renorming theorem is to characterize when the supremum in $(ii)$ of the previous theorem is obtained. That is, for every $\xi$ of countable cofinality, and in particular for every countable ordinal,  and every $1\leqslant q<\infty$, an example $X$ was given of a Banach space which has $1/\textbf{p}_\xi(X)+1/q=1$ but which did not admit any $\xi$-$q$-AUS (resp. $\xi$-AUF if $q=1$) norm. An example was also given in \cite{C3} for every ordinal $\xi$ and every $1\leqslant q<\infty$ of a Banach space $S$ which has $1/\textbf{p}_\xi(S)+1/q=1$ and which is $\xi$-$p$-AUS (resp. $\xi$-AUF if $q=1$).

\section{Combinatorical necessities}

\subsection{Trees of peculiar importance}

We first define some trees which will be of significant importance for us. Given a sequence $(\zeta_i)_{i=1}^n$ of ordinals and an ordinal $\zeta$, we let $\zeta+(\zeta_i)_{i=1}^n=(\zeta+\zeta_i)_{i=1}^n$. Given a set $G$ of sequences of ordinals and an ordinal $\zeta$, we let $\zeta+G=\{\zeta+t: t\in G\}$. For each $\xi\in \textbf{Ord}$ and $n\in \nn$, we define a tree $\Gamma_{\xi,n}$ which consists of decreasing ordinals in the interval $[0, \omega^\xi n)$. We let $$\Gamma_{\xi, 1}=\{(0)\}.$$  If $\xi$ is a limit ordinal and $\Gamma_{\zeta,1}$ has been defined for every $\zeta<\xi$, we let $$\Gamma_{\xi,1}=\bigcup_{\zeta<\xi}(\omega^\zeta+\Gamma_{\zeta+1,1}).$$  If for some $\xi$ and every $n\in \nn$, $\Gamma_{\xi,n}$ has been defined such that the first member of each sequence in $\Gamma_{\xi,n}$ lies in the interval $[\omega^\xi(n-1), \omega^\xi n)$, we let $$\Gamma_{\xi+1, 1}=\bigcup_{n=1}^\infty \Gamma_{\xi, n}.$$  Finally, if $\Gamma_{\xi, 1}$ has been defined, we let $\Lambda_{\xi,1,1}=\Gamma_{\xi, 1}$ and for $1<n\in \nn$ and $1\leqslant i\leqslant n$, we let $$\Lambda_{\xi, n, i}=\bigl\{(\omega^\xi(n-1)+t_1)\smallfrown \ldots \smallfrown (\omega^\xi (n-i)+t_i):  t_i\in \Gamma_{\xi,1}, t_1, \ldots, t_{i-1}\in MAX(\Gamma_{\xi,1})\bigr\}.$$ We refer to the sets $\Lambda_{\xi,n,1}, \ldots, \Lambda_{xi,n,n}$ as the \emph{levels} of $\Gamma_{\xi,n}$.

We also define $$\Lambda_{\xi, \infty, i}=\bigl\{t_1\smallfrown (\omega^\xi+t_2)\smallfrown \ldots \smallfrown (\omega^\xi(i-1)+t_i): t_i\in \Gamma_{\xi,1}, t_1, \ldots, t_{i-1}\in MAX(\Gamma_{\xi,1})\bigr\},$$  $$\Gamma_{\xi, \infty}=\bigcup_{i=1}^\infty \Lambda_{\xi, \infty, i}.$$

For a directed set $D$, an ordinal $\xi$, and $n\in \nn$, we let $$\Lambda_{\xi,n,i}.D = \{(\zeta_j, u_j)_{j=1}^k: (\zeta_i)_{i=1}^k\in \Lambda_{\xi,n}, u_i\in D\}.$$   We remark that for each $\zeta$, then for any directed set $D$, $(\omega^\zeta+\Gamma_{\zeta+1,1}).D$ is canonically identifiable with $\Gamma_{\zeta+1, 1}.D$. For any $\xi$ and any $n\in \nn$, $\Lambda_{\xi, n, 1}.D$ is canonically identifiable with $\Gamma_{\xi,1}.D$. Finally, for any $n\in \nn$, any ordinal $\xi$, and any $t\in MAX(\Lambda_{\xi, n,1})$ (resp. $t\in MAX(\Lambda_{\xi, \infty,1}.D)$),  $\{s\in \Gamma_{\xi, n+1}.D: t<s\}$ (resp. $\{s\in \Gamma_{\xi,\infty}.D\}$) is canonically identifiable with $\Gamma_{\xi,n}.D$ (resp. $\Gamma_{\xi,\infty}.D$).  We often implicitly use these canonical identifications without giving them specific names. 

We last define what it means for a subset of $\Gamma_{\xi,n}.D$ to be a \emph{unit}. For any ordinal $\xi$ and any $n\in \nn$, $\Lambda_{\xi,n,1}.D$ is a unit.  If for some $n\in \nn$, every ordinal $\xi$, and every $1\leqslant k\leqslant n$, the units in $\Gamma_{\xi,k}.D$ are defined, we say a subset $U$ of $\Gamma_{\xi, n+1}.D$ is a unit if either $U=\Lambda_{\xi, n+1, 1}.D$ or if there exists $t\in MAX(\Lambda_{\xi, n+1, 1}.D)$ such that, if $$j:\{s\in \Gamma_{\xi,n+1}.D: t<s\}\to \Gamma_{\xi,n}.D$$ is the canonical identification, $j(U)$ is a unit in $\Gamma_{\xi,n}.D$

\subsection{Cofinal and eventual sets}

For a fixed directed set $D$, we now define sets $\Omega_{\xi, n}$. Each set $\Omega_{\xi,n}$ will be a subset of the power set of $MAX(\Gamma_{\xi,n}.D)$.  Given $\mathcal{E}\subset \Gamma_{0,1}.D$, we can write $$\mathcal{E}=\{(0, u): u\in D_0\}$$ for some $D_0\subset D$. Then we say $\mathcal{E}\in \Omega_{0,1}$ if  $D_0$ is cofinal in $D$. 

Now suppose that for a limit ordinal $\xi$ and every $\zeta<\xi$, $\Omega_{\zeta+1, 1}$ has been defined. For each $\zeta<\xi$, let $j_\zeta:(\omega^\zeta+\Gamma_{\zeta+1}).D\to \Gamma_{\zeta+1,1}.D$ be the canonical identification. Then a subset $\mathcal{E}\subset MAX(\Gamma_{\xi, 1})$ lies in $\Omega_{\xi, 1}$ if  there exists a cofinal subset $M$ of $[0, \xi)$ such that for every $\zeta\in M$, $j_\zeta(\mathcal{E}\cap MAX((\omega^\zeta+\Gamma_{\zeta+1}.D))\in \Omega_{\zeta+1, 1}$. 

Now suppose that for an ordinal $\xi$ and every $n\in \nn$, $\Omega_{\xi,n}$ has been defined. Then we say $\mathcal{E}\subset MAX(\Gamma_{\xi+1, 1}.D)$ is a member of $\Omega_{\xi+1, 1}$ if there exists a cofinal subset $M$ of $\nn$ such that for every $n\in \nn$, $\mathcal{E}\cap \Gamma_{\xi, n}.D\in \Omega_{\xi,n}$. 

Last, suppose that for an ordinal $\xi$, a natural number $n$, and each $1\leqslant i\leqslant n$, $\Omega_{\xi,i}$ has been defined. Suppose that $\mathcal{E}\subset MAX(\Gamma_{\xi,n+1}.D)$ is given.   For each $t\in MAX(\Lambda_{\xi,n,1})$, let $P_t=\{s\in \Gamma_{\xi, n+1}.D: t<s\}$,  let $j_t:P_t\to \Gamma_{\xi,n}.D$ be the canonical identification, and let $j:\Lambda_{\xi, n,1}.D\to \Gamma_{\xi,1}.D$ be the canonical identification.  Let $$\mathcal{F}=\{t\in MAX(\Lambda_{\xi,n+1,1}.D): j_t(\mathcal{E}\cap MAX(P_t))\in \Omega_{\xi,n}\}.$$  Then we say $\mathcal{E}\in \Omega_{\xi, n+1}$ if $j(\mathcal{F})\in \Omega_{\xi,1}$.  

We remark that an easy induction proof shows that $MAX(\Gamma_{\xi,n}.D)\in \Omega_{\xi,n}$ for every $\xi$ an $n$, and if $\mathcal{F}\subset \mathcal{E}\subset MAX(\Gamma_{\xi, n}.D)$ and $\mathcal{F}\in \Omega_{\xi,n}$, then $\mathcal{E}\in \Omega_{\xi,n}$.

We refer to the sets in $\Omega_{\xi,n}$ as \emph{cofinal in} $\Gamma_{\xi,n}.D$.   We say a subset $\mathcal{E}$ of $MAX(\Gamma_{\xi,n}.D)$ is \emph{eventual} if $MAX(\Gamma_{\xi,n}.D)\setminus \mathcal{E}$ fails to be cofinal.  Each unit $U\subset \Gamma_{\xi,n}.D$ is canonically identifiable with  $\Gamma_{\xi,1}.D$, and as such we can define what it means for a subset of $MAX(U)$ to be cofinal or eventual using the identification with $\Gamma_{\xi,1}.D$.

We say a subset $B$ of $\cup_{n=1}^\infty MAX(\Lambda_{\xi, \infty, n}.D)$ is \begin{enumerate}[(i)]\item \emph{inevitable} provided that $B\cap MAX(\Lambda_{\xi, \infty, 1}.D)$ is eventual and for each $n\in \nn$ and each $t\in B\cap MAX(\Lambda_{\xi, \infty, n}.D)$, $B\cap MAX(\Lambda_{\xi, \infty, n+1}.D)\cap \Gamma_{\xi, \infty}(t<)$ is eventual, \item \emph{big} if it contains an inevitable subset. \end{enumerate}

We conclude this section by stating some combinatorial lemmas. We relegate the usually easy but somewhat technical proofs to the final section of the paper.  For the following proofs, we say $d:\Gamma_{\xi,n}.D\to \Gamma_{\xi,n}.D$ is a \emph{level map} if \begin{enumerate}[(i)]\item for any $\varnothing<s<t\in \Gamma_{\xi,n}.D$, $d(s)<d(t)$, \item if $U\subset \Lambda_{\xi,n,i}.D$ is a unit, then there exists a unit $V\subset \Lambda_{\xi,n,i}.D$ such that $d(U)\subset V$. \end{enumerate}  Note that since $\Gamma_{\xi,1}.D$ is a single unit, $(ii)$ is vacuous in the case $n=1$. Given a level map $d:\Gamma_{\xi,n}.D\to \Gamma_{\xi,n}.D$, we say $e:MAX(\Gamma_{\xi,n}.D)\to MAX(\Gamma_{\xi,n}.D)$ is an \emph{extension} of $d$ if for any $t\in MAX(\Gamma_{\xi,n}.D)$, $d(t)\leqslant e(t)$. Since $\Gamma_{\xi,n}.D$ is well-founded, any level map $d$ admits some extension.  We define an extension of a monotone map in the same way we define an extension of a level map. 

We let $\Pi(\Gamma_{\xi,n}.D)=\{(s,t)\in \Gamma_{\xi,n}.D\times MAX(\Gamma_{\xi,n}.D): s\leqslant t\}$.

\begin{lemma} Suppose that $\xi$ is an ordinal, $n\in \nn$, $X$ is a Banach space, and $(x_t)_{t\in \Gamma_{\xi,n}.D}$ is weakly null.  \begin{enumerate}[(i)]\item If $\mathcal{E}\subset MAX(\Gamma_{\xi,n}.D)$ is cofinal, there exists a level map $d:\Gamma_{\xi,n}.D\to \Gamma_{\xi,n}.D$ with extension $e$ such that $e(MAX(\Gamma_{\xi,n}.D))\subset \mathcal{E}$ and $(x_{d(t)})_{t\in \Gamma_{\xi,n}.D}$ is weakly null.   \item For any $k\in \nn$, if $MAX(\Gamma_{\xi,n}.D)\supset \mathcal{E}=\cup_{i=1}^k \mathcal{E}_i\in \Omega_{\xi,n}$, then there exists $1\leqslant j\leqslant k$ such that $\mathcal{E}_j\in \Omega_{\xi,n}$. \item If $F$ is a finite set and $\chi:\Pi(\Gamma_{\xi,n}.D)\to F$ is a function, then there exist a level map $d:\Gamma_{\xi,n}.D\to \Gamma_{\xi,n}.D$ with extension $e$ and $\alpha_1, \ldots, \alpha_n\in F$ such that for any $1\leqslant i\leqslant n$ and any $\Lambda_{\xi,n,i}.D\ni s\leqslant \in MAX(\Gamma_{\xi,n}.D)$, $\alpha_i=F(d(s), e(t))$, and such that $(x_{d(t)})_{t\in \Gamma_{\xi,n}.D}$ is weakly null.   \item If $h:\Pi(\Gamma_{\xi, n}.D)\to \rr$ is bounded and if $\mathcal{E}\subset MAX(\Gamma_{\xi,n}.D)$ is cofinal, then for any $\delta>0$,  there exist $a_1, \ldots, a_n\in \rr$ and a level map $d:\Gamma_{\xi,n}.D\to \Gamma_{\xi,n}.D$ with extension $e$ such that $e(MAX(\Gamma_{\xi,n}.D))\subset \mathcal{E}$, for each $1\leqslant i\leqslant n$ and each $\Lambda_{\xi,n,i}.D\ni s\leqslant t\in MAX(\Gamma_{\xi,n}.D)$, $h(d(s), e(t))\geqslant a_i-\delta$, and for any $t\in MAX(\Gamma_{\xi,n}.D)$, $\sum_{\varnothing<s\leqslant e(t)} \mathbb{P}_{\xi,n}(s)h(s, e(t)) \leqslant \delta+\sum_{i=1}^n a_i$.   \end{enumerate}

\label{stabilize}

\end{lemma}

\begin{rem}\upshape Items $(i)$ and $(ii)$ together yield that if $MAX(\Gamma_{\xi,n}.D)=\cup_{i=1}^k \mathcal{E}_i$, then there exists $1\leqslant j\leqslant k$ and a level map $d:\Gamma_{\xi,n}.D\to \Gamma_{\xi,n}.D$ with extension $e$ such that $(x_{d(t)})_{t\in \Gamma_{\xi,n}.D}$ is weakly null and $e(MAX(\Gamma_{\xi,n}.D))\subset \mathcal{E}_j$. A typical application of this result will be to have a real-valued function $h:MAX(\Gamma_{\xi,n}.D)\to C\subset \rr$, where $C$ is compact.  We may then fix $\delta>0$ and a finite cover $F_1, \ldots, F_k$ of $C$ by sets of diameter less than $\delta$. We then let $\mathcal{E}_i$ denote those $t\in MAX(\Gamma_{\xi,n}.D)$ such that $h(t)\in F_i$.  We may then find $d$, $e$, and $j$ as above and obtain $(x_{d(t)})_{t\in \Gamma_{\xi,n}.D}$ weakly null such that for every $t\in MAX(\Gamma_{\xi,1}.D)$, $h(e(t))\in F_j$.

Similarly, we will often apply $(iii)$ to a function $h_1:\Pi(\Gamma_{\xi,n}.D)\to C\subset \rr$, where $C$ is compact, by first covering $C$ by $F_1, \ldots, F_k$ of sets of diameter less than $\delta$. We then define $h(s,t)$ to be the minimum $j\leqslant k$ such that $h_1(s,t)\in F_j$.

\end{rem}

\begin{corollary} Given a non-zero $A:X\to Y$ and an ordinal $\xi$, let  $$\overline{\varrho}_\xi(\sigma,A)=\sup\Bigl\{\inf\{\|y+Ax\|-1: t\in T.D, x\in \text{\emph{co}}(x_s: \varnothing<s\leqslant t)\}\Bigr\},$$ where the supremum is taken over all $y\in S_Y$, all weakly null collections $(x_t)_{T.D}\subset \sigma B_X$ with $o(T)=\omega^\xi$.  Then $\overline{\varrho}_\xi(\sigma,A)=\varrho_\xi(\sigma,A)$.

\end{corollary}

\begin{proof} Since $\varrho_\xi(\sigma, A)$ involves taking the supremum over $y\in B_Y$ and $\overline{\varrho}_\xi(\sigma,A)$ involves taking the supremum only over $y\in S_Y$, $\overline{\varrho}_\xi(\sigma,A)\leqslant \varrho_\xi(\sigma,A)$. To obtain a contradiction, assume there exists $y\in B_Y$, $\delta>0$, and a collection $(x_t)_{t\in T.D}\subset \sigma B_X$ such that $$\inf_{t\in T.D} \|y+Ax\| >1+\overline{\varrho}_\xi(\sigma, A).$$  We may assume $T=\Gamma_{\xi,1}$. By perturbing, we may also assume $y\neq 0$.  For each $t\in MAX(\Gamma_{\xi,1}.D)$, fix $y^*_t\in B_{Y^*}$ such that $\text{Re\ }y^*_t(y+A\sum_{\varnothing<s\leqslant t}\mathbb{P}_{\xi,1}(s) x_s)=\|y+A\sum_{\varnothing<s\leqslant t}\mathbb{P}_{\xi,1}(s)x_s\|$.    By Lemma \ref{stabilize}, we may fix  a monotone map $d:\Gamma_{\xi,1}.D\to \Gamma_{\xi,1}.D$ with extension $e$ such that $(x_{d(s)})_{s\in\Gamma_{\xi,1}.D)}$ is still weakly null and $$1+\overline{\varrho}_\xi(\sigma, A)< \inf_{(s,t)\in \Pi(\Gamma_{\xi,1}.D)} \text{Re\ }y^*_{e(t)}(y+Ax_{d(s)}).$$     Let $\mathcal{E}=\{t\in MAX(\Gamma_{\xi,1}.D): \text{Re\ }y^*_{e(t)}(y)\geqslant 0\}$ and  let $\mathcal{F}=MAX(\Gamma_{\xi,1}.D)\setminus \mathcal{E}$.  Then by switching $y$ with $-y$ if necessary, we may choose another monotone map $d':\Gamma_{\xi,1}.D\to \Gamma_{\xi,1}.D$ with extension $e'$ such that $(x_{d\circ d'(s)})_{s\in \Gamma_{\xi,1}.D}$ is weakly null, $\text{Re\ }y^*_{e\circ e'(t)}(y)\geqslant 0$ for all $t\in MAX(\Gamma_{\xi,1}.D)$, and $$\inf_{(s,t)\in \Pi(\Gamma_{\xi,1}.D)} \text{Re\ }y^*_{e\circ e'(t)}(y+Ax_{d\circ d'(t)})> 1+\overline{\varrho}_\xi(\sigma, A).$$  Then \begin{align*} \inf\{\|\frac{y}{\|y\|}+Ax\| : t\in MAX(\Gamma_{\xi,1}.D), x\in \text{co}(x_{d\circ d'(s)}: \varnothing<s\leqslant t)\}  &  \geqslant \inf\{\text{Re\ }y^*_{e\circ e'(t)}(\frac{y}{\|y\|}+Ax) : t\in MAX(\Gamma_{\xi,1}.D), x\in \text{co}(x_{d\circ d'(s)}: \varnothing<s\leqslant t)\} \\ & \geqslant \inf\{\text{Re\ }y^*_{e\circ e'(t)}(y+Ax) : t\in MAX(\Gamma_{\xi,1}.D), x\in \text{co}(x_{d\circ d'(s)}: \varnothing<s\leqslant t)\} \\ & > 1+\overline{\varrho}_\xi(\sigma, A).\end{align*} This contradiction finishes the proof.

\end{proof}

\begin{corollary} Let $A:X\to Y$ be a Banach space, $D$ a weak neighborhood basis at $0$ in $X$, and $\xi$ an ordinal. \begin{enumerate}[(i)]\item For any $\tau\geqslant 0$, $\varrho_\xi(\sigma, A)\leqslant \tau$ if and only for every $\sigma>0$, every  weakly null $(x_t)_{t\in \Gamma_{\xi,1}.D}\subset \sigma B_X$ and every $y\in B_Y$,  $$\inf_{t\in \Gamma_{\xi,1}.D} \|y+A\sum_{\varnothing<s\leqslant t}\mathbb{P}_{\xi,1}(s)x_s\| \leqslant 1+\tau,$$ if and only if for every $\sigma>0$, every  weakly null $(x_t)_{t\in \Gamma_{\xi,1}.D}\subset \sigma B_X$,  every $y\in B_Y$, and every $\tau'>0$, $$\{t\in MAX(\Gamma_{\xi,1}.D): \|y+A\sum_{\varnothing<s\leqslant t} \mathbb{P}_{\xi,1}(s) x_s\|\leqslant 1+\tau'\}$$ is eventual. \item For any $\tau\geqslant 0$ and any $1<p<\infty$, $\textbf{\emph{t}}_{\xi,p}(A)\leqslant \tau$ if and only if for every $\sigma>0$, every  weakly null $(x_t)_{t\in \Gamma_{\xi,1}.D}\subset \sigma B_X$, and every $y\in Y$, $$\inf_{t\in MAX(\Gamma_{\xi,1}.D} \|y+A\sum_{\varnothing<s\leqslant t} \mathbb{P}_{\xi,1}(s) x_s\|^p \leqslant \|y\|^p +\tau^p \sigma^p$$ if and only if for  every $\sigma>0$, every  weakly null $(x_t)_{t\in \Gamma_{\xi,1}.D}\subset \sigma B_X$, every $y\in B_Y$, and every  $\tau'> \|y\|^p+\tau^p \sigma^p,$ $$\{t\in MAX(\Gamma_{\xi,1}.D): \|y+A\sum_{\varnothing<s\leqslant t} \mathbb{P}_{\xi,1}(s) x_s\|^p\leqslant \tau'\}$$ is eventual. \item For any $\tau\geqslant 0$, $\textbf{\emph{t}}_{\xi,p}(A)\leqslant \tau$ if and only if for every $\sigma>0$, every  weakly null $(x_t)_{t\in \Gamma_{\xi,1}.D}\subset \sigma B_X$, and every $y\in Y$, $$\inf_{t\in MAX(\Gamma_{\xi,1}.D} \|y+A\sum_{\varnothing<s\leqslant t} \mathbb{P}_{\xi,1}(s) x_s\| \leqslant \max\{\|y\|, \tau \sigma\},$$ if and only if for every $\sigma>0$, every  weakly null $(x_t)_{t\in \Gamma_{\xi,1}.D}\subset \sigma B_X$,  every $y\in B_Y$,  $\tau'> \max\{\|y\|, \tau \sigma\},$ $$\{t\in MAX(\Gamma_{\xi,1}.D): \|y+A\sum_{\varnothing<s\leqslant t} \mathbb{P}_{\xi,1}(s) x_s\|^p\leqslant \tau'\}$$ is eventual.\end{enumerate}

\end{corollary}

\begin{proof} We prove only $(i)$, with $(ii)$ and $(iii)$ being similar. Fix $y\in B_Y$ and $(x_t)_{t\in \Gamma_{\xi,1}.D}\subset \sigma B_X$ and assume that for some $\tau'>\tau$,  $$\mathcal{E}:=\{t\in MAX(\Gamma_{\xi,1}.D):\|y+A\sum_{\varnothing<s\leqslant t}\mathbb{P}_{\xi,1}(s) x_s\|\leqslant 1+\tau'\}$$ fails to be eventual.   Fix $\delta>0$ such that $\tau'-2\delta>\tau$. For each $t\in MAX(\Gamma_{\xi,1}.D)$, fix $y^*_t\in B_{Y^*}$ such that $$\|y+A\sum_{\varnothing<s\leqslant t} \mathbb{P}_{\xi,1}(s)x_s\|=\text{Re\ }y^*_t(y+A\sum_{\varnothing<s\leqslant t} \mathbb{P}_{\xi,1}(s)x_s)> 1+\tau'.$$  Then by Lemma \ref{stabilize}$(iii)$ applied to the function $h(s,t)=\text{Re\ }y^*_t(y+Ax_s)$, there exist $a\in \rr$ and a monotone map $d:\Gamma_{\xi,1}.D\to \Gamma_{\xi,1}.D$ with extension $e:MAX(\Gamma_{\xi,1}.D)\to \mathcal{E}$ such that $(x_{d(t)})_{t\in \Gamma_{\xi,1}.D}$ is weakly null, for every $(s,t)\in \Pi(\Gamma_{\xi,1}.D)$, $\text{Re\ }y^*_{e(t)}(y+Ax_{d(t)}) \geqslant a-\delta$, and for each $t\in MAX(\Gamma_{\xi,1}.D)$, $$a+\delta\geqslant \sum_{\varnothing<s\leqslant e(t)} \mathbb{P}_{\xi,1}(s)[\text{Re\ }y^*_{e(t)}(y+Ax_s)]= \text{Re\ }y^*_{e(t)}\bigl(y+\sum_{\varnothing<s\leqslant e(t)}\mathbb{P}_{\xi,1}(s)x_s\bigr)\geqslant 1+\tau'.$$  This means $a-\delta > 1+\tau$. But for any $t\in MAX(\Gamma_{\xi,1}.D)$ and any convex combination $x$ of $(x_{d(s)}: \varnothing<s\leqslant t)$, $$\|y+Ax \| \geqslant \text{Re\ }y^*_{e(t)}(y+Ax) \geqslant a-\delta.$$ Thus $y$, $(x_{d(t)})_{t\in \Gamma_{\xi,1}.D}$ witness that $\varrho_\xi(\sigma, A)\geqslant a-\delta>\tau$.

Now suppose $\varrho_\xi(\sigma, A)>\tau'>\tau$. Then there exists some tree $T$ with $o(T)=\omega^\xi$, a directed set $D_1$, $y\in B_Y$, and a weakly null $(x_t)_{t\in T.D_1}\subset \sigma B_X$ such that $$\inf \{\|y+Ax\|: t\in T.D_1, x\in \text{co}(x_s: \varnothing<s\leqslant t)\} >1+\tau'.$$  Since $o(T)=o(\Gamma_{\xi,1})$, there exists a monotone, length-preserving map $\theta:\Gamma_{\xi,1}\to T$. Fix $t\in MIN(\Gamma_{\xi,1})$ and consider the weakly null net $(x_{(\theta(t), u)})_{u\in D_1}$. We may fix for each $u\in D$ some $v_u\in D_1$ such that $x_{(\theta(t), v_u)}\in u$. Let $\Theta((t, u))=(\theta(t), v_u)$.   Now suppose that for some $t\in \Gamma_{\xi,1}.D\setminus MIN(\Gamma_{\xi,1}.D)$, $\Theta(t^-)$ has been defined such that if $t^-=(\zeta_i, u_i)_{i=1}^n$ and $\theta((\zeta_i)_{i=1}^n)=(\mu_i)_{i=1}^n$, then $\Theta(t^-)=(\mu_i, v_i)_{i=1}^n$ for some $v_1, \ldots, v_n$. We then write $t=(\zeta_i, u_i)_{i=1}^{n+1}$ and $\theta(t)=(\mu_i)_{i=1}^{n+1}$ for some $\zeta_{n+1}$, $\mu_{n+1}$, $u_{n+1}$.  Then consider the weakly null net $(x_{\Theta(t^-)\smallfrown (\mu_{n+1}, v)})_{v\in D_1}$. Choose some $v_{n+1}\in D_1$ such that $x_{\Theta(t^-)\smallfrown (\mu_{n+1},v_{n+1})}\in u_{n+1}$ and define $\Theta(t)=\Theta(t^-)\smallfrown (\mu_{n+1}, v_{n+1})$. Now let $x'_t=x_{\Theta(t)}$ for $t\in \Gamma_{\xi,1}.D$. It is clear that $(x_t')_{t\in \Gamma_{\xi,1}.D}\subset \sigma B_X$ is weakly null and satisfies $$\inf_{t\in \Gamma_{\xi,1}.D} \|y+A\sum_{\varnothing<s\leqslant t}\mathbb{P}_{\xi,1}(s) x_s'\|\geqslant 1+\tau'.$$  

Now suppose that $y\in B_Y$, $(x_t)_{t\in \Gamma_{\xi,1}.D}\subset \sigma B_X$ is a weakly null collection such that $$\inf_{t\in MAX(\Gamma_{\xi,1}.D)} \|y+\sum_{\varnothing<s\leqslant t}\mathbb{P}_{\xi,1}(s) x_s\| >1+\tau'>1+\tau.$$  Then $$\{t\in MAX(\Gamma_{\xi,1}.D): \|y+A\sum_{\varnothing<s\leqslant t}\mathbb{P}_{\xi,1}(s)x_s\|\leqslant 1+\tau'\}=\varnothing,$$ and therefore cannot be eventual.

\end{proof}

\begin{corollary} Let $\xi$ be an ordinal,  $A:X\to Y$ an operator, and $1<p< \infty$.  Then if $\textbf{\emph{t}}_{\xi,p}(A)\leqslant C$, $G\subset Y$ is compact, $T\subset \mathbb{K}$ is compact,  $\ee>0$, and $(x_t)_{t\in \Gamma_{\xi,1}.D}\subset B_X$ is weakly null, $$\Bigl\{t\in MAX(\Gamma_{\xi,1}.D): (\forall y\in G)(\forall \alpha\in T)(\|y+A\sum_{\varnothing<s\leqslant t}\alpha \mathbb{P}_{\xi,1}(s)x_s\|^p\leqslant \|y\|^p+C^p \alpha^p+\ee)\Bigr\}$$ is eventual. 

The analogous statement holds when $p=\infty$ if we replace the $\ell_p$ norm with the maximum. 

\label{compactness corollary}
\end{corollary}

\begin{proof} Fix $\delta>0$, a finite $\delta$-net $F\subset G$, and a finite $\delta$-net $S\subset T$. Then for each $y\in F$ and each $\alpha\in F$, $$\mathcal{E}_{y, \alpha}:=\Bigl\{t\in MAX(\Gamma_{\xi,1}.D): \|y+A\sum_{\varnothing<s\leqslant t}\alpha\mathbb{P}_{\xi,1}(s)x_s\|^p\leqslant \|y\|^p+C^p \alpha^p+\delta\Bigr\}$$ is eventual. Thus $\mathcal{E}:=\cap_{y\in F, \alpha\in S} \mathcal{E}_{y, \alpha}$ is also eventual, being a finite intersection of eventual sets. Provided $\delta>0$ is chosen sufficiently small, $$\mathcal{E}\subset \Bigl\{t\in MAX(\Gamma_{\xi,1}.D): (\forall y\in G)(\forall \alpha\in T)(\|y+A\sum_{\varnothing<s\leqslant t}\alpha\mathbb{P}_{\xi,1}(s)x_s\|^p\leqslant \|y\|^p+C^p \alpha^p+\ee)\Bigr\},$$ whence the latter set is eventual.

\end{proof}

\begin{corollary} Let $X,Y$ be Banach spaces, $\alpha, \beta>0$ with $\alpha+\beta=1$, $A,B:X\to Y$, and $\sigma>0$. Then  $\varrho_\xi(\sigma, \alpha A+\beta B)\leqslant \alpha \varrho_\xi(\sigma, A)+\beta \varrho_\xi(\sigma, B).$

\label{convexity corollary1}
\end{corollary}

\begin{proof} For any $y\in B_Y$, any weakly null $(x_t)_{t\in \Gamma_{\xi,1}.D}\subset \sigma B_X$, any $a>\varrho_\xi(\sigma, A)$, and $b>\varrho_\xi(\sigma, B)$, $$\mathcal{E}_A:=\{t\in MAX(\Gamma_{\xi,1}.D): \|y+A\sum_{\varnothing<s\leqslant t}\mathbb{P}_{\xi,1}(s) x_s\|-1 \leqslant a\}$$ and $$\mathcal{E}_B:= \{t\in MAX(\Gamma_{\xi,1}.D): \|y+B\sum_{\varnothing<s\leqslant t}\mathbb{P}_{\xi,1}(s)x_s\|-1\leqslant b\}$$ are eventual. This means there exists $t\in \mathcal{E}_A\cap \mathcal{E}_B$, whence \begin{align*} \inf_{u\in MAX(\Gamma_{\xi,1}.D)} \|y+(\alpha A+\beta B)\sum_{\varnothing<s\leqslant u}\mathbb{P}_{\xi,1}(s) x_s\|-1  & \leqslant \alpha \Bigl[\|y+A\sum_{\varnothing<s\leqslant t} \mathbb{P}_{\xi,1}(s) x_s\|-1\Bigr] \\ & + \beta \Bigl[\|y+B\sum_{\varnothing<s\leqslant t} \mathbb{P}_{\xi,1}(s) x_s \|-1\Bigr] \\ & \leqslant \alpha a+\beta b. \end{align*} Since $a>\varrho_\xi(\sigma, A)$, $b>\varrho_\xi(\sigma, B)$, $y\in B_Y$, and $(x_t)_{t\in \Gamma_{\xi,1}.D}\subset \sigma B_X$ weakly null were arbitrary, we are done.

\end{proof}

\begin{corollary} Suppose that $A:X\to Y$ is an operator, $D$ a directed set,  $G\subset B_Y$ is such that $\overline{\text{\emph{co}}}(G)=B_Y$, $\sigma>0$, $\tau\geqslant 0$, and for every $y\in G$ and every weakly null collection $(x_t)_{t\in \Gamma_{\xi,1}.D}\subset \sigma B_X$, $$\inf_{t\in MAX(\Gamma_{\xi,1}.D} \|y+A\sum_{\varnothing<s\leqslant t}\mathbb{P}_{\xi,1}(s) x_s\|-1\leqslant \tau.$$   Then $\varrho_\xi(\sigma, A)\leqslant \tau$.

\label{convexity corollary2}
\end{corollary}

\begin{proof} By Lemma \ref{stabilize}, the hypotheses imply that for any $\tau'>\tau$, any $y\in G$, and any $(x_t)_{t\in \Gamma_{\xi,1}.D}\subset \sigma B_X$ weakly null, $$\{t\in MAX(\Gamma_{\xi,1}.D): \|y+A\sum_{\varnothing<s\leqslant t}\mathbb{P}_{\xi,1}(s) x_s\|-1\leqslant \tau'\}$$ is eventual.  For any $y\in B_Y$ and any weakly null $(x_t)_{t\in \Gamma_{\xi,1}.D}\subset \sigma B_X$, we may first fix $\delta>0$, positive scalars $a_1, \ldots, a_n$ summing to $1$, and $y_1, \ldots, y_n\in G$ such that $\|y-\sum_{i=1}^n a_iy_i\|\leqslant \delta$.    We then note that for each $1\leqslant i\leqslant n$, $$\mathcal{E}_i:= \{t\in MAX(\Gamma_{\xi,1}.D): \|y_i+A\sum_{\varnothing<s\leqslant t}\mathbb{P}_{\xi,1}(s) x_s\|-1 \leqslant \tau+\delta\}$$ is eventual.  Then we may fix any $t\in \cap_{i=1}^n \mathcal{E}_i$ and note that for this $t$, $$\|y+A\sum_{\varnothing<s\leqslant t} \mathbb{P}_{\xi,1}(s) x_s\|-1 \leqslant \delta+\sum_{i=1}^n a_i [\|y_i+A\sum_{\varnothing<s\leqslant t}\mathbb{P}_{\xi,1}(s) x_s\|-1 ] \leqslant 2\delta+\tau.$$  Since $\delta>0$ was arbitrary, $$\inf_{t\in MAX(\Gamma_{\xi,1}.D)} \|y+A\sum_{\varnothing<s\leqslant t}\mathbb{P}_{\xi,1}(s) x_s\|-1 \leqslant \tau.$$  Since $y\in B_Y$ and $(x_t)_{t\in \Gamma_{\xi,1}.D}\subset \sigma B_X$ weakly null were arbitrary, we deduce that $\varrho_\xi(\sigma, A)\leqslant \tau$.

\end{proof}

\section{Upper tree estimates and renormings} 

Given a collection $(x_t)_{t\in \Gamma_{\xi, \infty}.D}\subset X$ and $t\in MAX(\Lambda_{\xi, \infty, n}.D)$, we let $z_1^t, \ldots, z_n^t$ be the vectors given by $$z^t_i= \sum_{\Lambda_{\xi, \infty, i}.D\ni s\leqslant t} \mathbb{P}_{\xi,\infty}(s) x_s.$$  This is a convex block sequence of $(x_s)_{\varnothing<s\leqslant t}$.   We note that this notation should reference the underlying collection $(x_t)_{t\in \Gamma_{\xi, \infty}.D}$, but having no such reference will cause no confusion. 

Given a Banach space $E$ with basis $(e_i)_{i=1}^\infty$, $C>0$, an ordinal $\xi$, and an operator $A:X\to Y$, we say $A$ satisfies $C$-$\xi$-$E$ upper tree estimates provided that for any normally weakly null $(x_t)_{t\in \Gamma_{\xi, \infty}.D}\subset B_X$, $$\bigcup_{n=1}^\infty \Bigl\{t\in MAX(\Lambda_{\xi, \infty, n}.D):(\forall (a_i)_{i=1}^n\in \mathbb{K}^n)\Bigl(\|A\sum_{i=1}^n z^t_i\| \leqslant C\|\sum_{i=1}^n a_i e_i\|_E\Bigr)\Bigr\}$$ is big. We say $A$ satisfies $\xi$-$E$ upper tree estimates provided that there exists $C>0$ such that $A$ satisfies $C$-$\xi$-$E$ upper tree estimates.

\begin{theorem} Fix an ordinal $\xi$ and an operator $A:X\to Y$. \begin{enumerate}[(i)]\item $A$ is $\xi$-$p$ asymptotically uniformly smoothable if and only if $A$ satisfies $\xi$-$\ell_p$ upper tree estimates. \item $A$ is $\xi$-asymptotically uniformly flattenable if and only if $A$ satisfies $\xi$-$c_0$ upper tree estimates. \end{enumerate}

\label{renorming}
\end{theorem}

\begin{proof}$(i)$ First assume that $A$ is $\xi$-$p$-AUS and fix $C>\textbf{t}_{\xi,p}(A)$.  Fix positive numbers $(\ee_n)_{n=1}^\infty$ such that $\textbf{t}_{\xi,p}(A)^p+\sum_{n=1}^\infty \ee_n < C^p$.   Fix $(x_t)_{t\in \Gamma_{\xi,n}.D}\subset B_X$ weakly null. For $n\in \nn$, let $B_n$ denote the set of all $t\in MAX(\Lambda_{\xi,\infty, n}.D)$ such that for every $(a_i)_{i=1}^n\in B_{\ell_p^n}$, $$\|A\sum_{i=1}^n a_i z_i^t\|^p\leqslant \|A\sum_{i=1}^{n-1} a_i z_i^t\|^p+ \textbf{t}_{\xi,p}(A)^p|a_n|^p+ \ee_n.$$   Let $B_n'$ be the set consisting of those $t\in MAX(\Lambda_{\xi, \infty, n}.D)$ such that for each $1\leqslant m\leqslant n$, if $s\leqslant t$ is such that $s\in MAX(\Lambda_{\xi, \infty, m}.D)$, then $s\in B_m$.  It follows from Corollary \ref{compactness corollary} that  $B:=\cup_{n=1}^\infty B_n'$  is inevitable.  Furthermore, $B$ is contained in $$\bigcup_{n=1}^\infty \Bigl\{t\in MAX(\Lambda_{\xi, \infty, n}.D): (\forall (a_i)_{i=1}^n\in \mathbb{K}^n)\Bigl(\|A\sum_{i=1}^n a_i z_i^t\|\leqslant C\|(a_i)_{i=1}^n\|_{\ell_p^n}\Bigr)\Bigr\},$$ so the latter set must be big. This shows that $A$ satisfies $C$-$\xi$-$\ell_p$ upper tree estimates.

Now assume $A:X\to Y$ satisfies $C_1$-$\xi$-upper tree estimates for some $C_1>0$ and let $C=2C_1$. Fix a weak neighborhood basis $D$ at $0$ in $X$. Define $[\cdot]$ on $Y$ by letting $[y]$ be the infimum of those $\mu>0$ such that for any weakly null $(x_t)_{t\in \Gamma_{\xi, \infty}.D}\subset B_X$, $$\bigcup_{n=1}^\infty \Bigl\{t\in MAX(\Gamma_{\xi,n}.D): (\forall (a_i)_{i=1}^n\in \mathbb{K}^n)\Bigl(\bigl(\frac{1}{C^p}\|y+A\sum_{i=1}^n a_i z_i^t\|^p - \sum_{i=1}^n |a_i|^p\bigr)^{1/p} \leqslant \mu\Bigr)\Bigr\}$$ is big. We note that $\frac{1}{C}\|y\| \leqslant [y]\leqslant \|y\|$ and $[cy]=|c|[y]$ for any scalar $c$. Let $G=\{y\in Y: [y]<1\}$.    We let $|y|=\inf\bigl\{\sum_{i=1}^n [y_i]: n\in \nn, \sum_{i=1}^n y_i=y\bigr\}$.  Then $|\cdot|$ is a norm on $Y$ such that $\frac{1}{C}\|y\|\leqslant |y|\leqslant \|y\|$. We will show that for any $y\in G$ and any weakly null $(x_t)_{t\in \Gamma_{\xi,1}.D}\subset \sigma B_X$, \begin{align*} \inf_{t\in MAX(\Gamma_{\xi,1}.D)} |y+A\sum_{\varnothing<s\leqslant t} \mathbb{P}_{\xi,1}(s) x_s| & \leqslant   \inf_{t\in MAX(\Gamma_{\xi,1}.D)} [y+A\sum_{\varnothing<s\leqslant t} \mathbb{P}_{\xi,1}(s) x_s] \\ &  \leqslant (1+\sigma^p)^{1/p}  \leqslant 1+\sigma^p,\end{align*} whence $\varrho_\xi(\sigma, A:X\to (Y, |\cdot|))\leqslant \sigma^p$ for all $\sigma>0$ by Corollary \ref{convexity corollary2}. For this we are using the fact that $\overline{\text{co}}(G)=B_Y^{|\cdot|}$. To that end, suppose $\mu\geqslant 0$, $y\in G$, and $(x_t)_{t\in \Gamma_{\xi,1}.D}\subset B_X$ are such that $$1+\mu^p<\inf_{t\in \Gamma_{\xi,1}.D} |y+A\sum_{\varnothing<s\leqslant t}\mathbb{P}_{\xi,1}(s) \sigma x_s|^p \leqslant \inf_{t\in \Gamma_{\xi,1}.D} [y+A\sum_{\varnothing<s\leqslant t}\mathbb{P}_{\xi,1}(s)\sigma x_s]^p.$$ Then for each $t\in MAX(\Gamma_{\xi,1}.D)$, there exists $(x_{s,t})_{s\in \Gamma_{\xi,\infty}.D}\subset B_X$ weakly null such that \begin{align*} B_t:=\bigcup_{n=1}^\infty &\Bigl\{s\in MAX(\Lambda_{\xi,\infty, n}.D) : \\ & (\forall (a_i)_{i=1}^n\in \mathbb{K}^n)\Bigl( \frac{1}{C^p}\|y+A \sum_{\varnothing<u\leqslant t}\mathbb{P}_{\xi,1}(u)\sigma x_u + A\sum_{i=1}^n a_iz^{s,t}_i\|^p- \sum_{i=1}^n |a_i|^p \leqslant 1+\mu^p\Bigr)\Bigr\}\end{align*} fails to be big. Here, $(z^{s,t}_i)_{i=1}^n$ is the obvious convex blocking of $(x_{u, t})_{\varnothing<u\leqslant s}$.

 For $t\in MAX(\Gamma_{\xi,1}.D)$, let $\phi_t:\Gamma_{\xi,\infty}.D\to \{s\in \Gamma_{\xi,\infty}.D: t<s\}$ be given by $$\phi_t( t_1\smallfrown (\omega^\xi+t_2)\smallfrown \ldots \smallfrown (\omega^\xi(n-1)+t_n))= t\smallfrown (\omega^\xi +t_1)\smallfrown (\omega^\xi 2+t_2)\smallfrown \ldots \smallfrown (\omega^\xi n+t_n).$$  This is the canonical identification of $\Gamma_{\xi,\infty}.D$ with $\{s\in \Gamma_{\xi,\infty}.D: t<s\}$. Let us now extend $(x_t)_{t\in \Gamma_{\xi,1}.D}$ to a collection $(x_t)_{t\in \Gamma_{\xi,\infty}.D}$ by letting $x_s= x_{\phi_t^{-1}(s), t}$ if $t\in MAX(\Gamma_{\xi, 1}.D)$ is such that $t<s$.  Since $[y]<1$, there exists an inevitable set $B$ which is contained in $$\bigcup_{n=1}^\infty \Bigl\{t\in MAX(\Lambda_{\xi,\infty, n}.D): (\forall (a_i)_{i=1}^n\in \mathbb{K}^n)\Bigl(\frac{1}{C^p}\|y+A\sum_{i=1}^n a_iz^t_i\|^p-\sum_{i=1}^n |a_i|^p\leqslant 1\Bigr)\Bigr\}.$$  Fix any $t\in MAX(\Lambda_{\xi,\infty, 1}.D)\cap B$ and note that $$B':=\{s\in \Gamma_{\xi,\infty}.D: \phi_t(s)\in B\}$$ must be big. Since $B_t$ is not big, there exists $s\in B'\setminus B_t$.  Since $s\notin B_t$, if $v=\phi_t(s)$, there exists $(a_i)_{i=1}^n\in \mathbb{K}^n$ such that \begin{align*} \frac{1}{C^p}\|y+A \sigma z^v_1+\sum_{i=2}^{n+1} a_{i-1} z_i^v\|^p-\sum_{i=1}^n |a_i|^p & = \frac{1}{C^p}\|y+A\sum_{\varnothing<u\leqslant t}\mathbb{P}_{\xi,1}(u) \sigma x_s + A\sum_{i=1}^n a_iz_i^{s,t}\|^p -\sum_{i=1}^n |a_i|^p \\ & \geqslant 1+\mu^p.\end{align*}  But since $s\in B'$, for this $(a_i)_{i=1}^n$, $$\frac{1}{C^p}\|y+A\sigma z^v_1+\sum_{i=2}^{n+1} a_{i-1} z_i^v\|^p-\sigma^p-\sum_{i=1}^n |a_i|^p\leqslant 1.$$  This shows that $\mu\leqslant \sigma$ and finishes $(i)$.

$(ii)$ The same argument as in $(i)$ with $\ell_p^n$ replaced by $\ell_\infty^p$ yields that if $A$ is $\xi$-AUF, $A$ satisfies $C$-$\xi$-$c_0$ upper tree estimates for any $C>\textbf{t}_{\xi, \infty}(A)$.

Now assume $A:X\to Y$ satisfies $C$-$\xi$-$c_0$ upper tree estimates. Fix a weak neighborhood basis $D$ at $0$ in $X$. Define $g$ on $Y$ by letting $g(y)$ be the infimum of those $C_1>0$ such that for any weakly null $(x_t)_{t\in \Gamma_{\xi, \infty}.D}\subset B_X$, $$B_y(C_1):= \bigcup_{n=1}^\infty \Bigl\{t\in MAX(\Lambda_{\xi, \infty, n}.D): \|y+A\sum_{i=1}^n  z^t_i\|\leqslant C_1\Bigr\}$$ is big.  Then we claim that $g$ has the following properties: \begin{enumerate}[(a)]\item $g$ is $1$-Lipschitz. \item $g$ is balanced. \item for any $y\in Y$, $\|y\|\leqslant g(y)\leqslant \|y\|+C$. \item $g$ is convex. \end{enumerate}  The first three properties are obvious from definitions.   To see convexity, fix $y, y'\in Y$, $\alpha, \alpha'>0$ with $\alpha+\alpha'=1$, and $\beta>g(y)$, $\beta'>g(y')$.  Fix a weakly null $(x_t)_{t\in \Gamma_{\xi,\infty}.D}\subset B_X$ and note that by the triangle inequality, $B_{\alpha y+\alpha' y}(\alpha \beta +\alpha' \beta')\supset B_y(\beta)\cap B_{y'}(\beta')$, where $B_{\alpha y+\alpha'y'}(\alpha \beta+\alpha'\beta')$, $B_y(\beta)$, $B_{y'}(\beta')$ are as in the definition of $|\cdot|$.   Since supersets of big sets are big, $B_y(\beta)$, $B_{y'}(\beta')$ are big, and the intersection of two big sets is big, $B_{\alpha y+\alpha' y'}(\alpha\beta+\alpha' \beta)$ is big. From this we deduce convexity. 

Let $G=\{y\in Y: g(y)\leqslant 1+C\}$ and let $|\cdot|$ be the Minkowski functional of $G$. Then if $y\in B_Y$, $g(y)\leqslant \|y\|+C\leqslant 1+C$, and $y\in G$. If $y\in G$, $\|y\|\leqslant g(y)\leqslant 1+C$, so $y\in (1+C)B_Y$. Thus $|\cdot|$ is an equivalent norm on $Y$.  We will show that $\varrho_\xi(1, A:X\to (Y, |\cdot|))=0$.  To that end, fix $y\in Y$ with $|y|<1$ and $(x_t)_{t\in \Gamma_{\xi,1}.D}\subset B_X$ weakly null.   Assume $$\inf_{t\in MAX(\Gamma_{\xi,1}.D)} |y+A\sum_{\varnothing<s\leqslant t}\mathbb{P}_{\xi,1}(s)x_s|>1,$$ which means there exists $\mu>0$ such that $$\inf_{t\in MAX(\Gamma_{\xi,1}.D)} g(y+A\sum_{\varnothing<s\leqslant t} \mathbb{P}_{\xi,1}(s) x_s)>1+C+\mu.$$  For each $t\in MAX(\Gamma_{\xi,1}.D)$, we may fix $(x^{s,t})_{s\in \Gamma_{\xi, \infty}.D}\subset B_X$ weakly null such that $$B_t:=\bigcup_{n=1}^\infty \Bigl\{s\in MAX(\Lambda_{\xi, \infty, n}.D):\|y+A\sum_{\varnothing<s\leqslant t} z^{s,t}_i\|\leqslant 1+C+\mu\Bigr\}$$ fails to be big and define $(x_t)_{t\in \Gamma_{\xi,\infty}.D}$ as in $(i)$.    We now argue as in $(i)$ to find some inevitable subset $B$ of $$\bigcup_{n=1}^\infty \Bigl\{t\in MAX(\Lambda_{\xi,\infty, n}.D): \|y+A\sum_{i=1}^n z_i^t\|\leqslant 1+C\Bigr\},$$ some $t\in MAX(\Lambda_{\xi, \infty, 1}.D)\cap B$, and $s$ such that $\phi_t(s)\in B'\setminus B_t$, where $B'=\{u\in \Gamma_{\xi, \infty}.D: \phi_t(u)\in B\}$.  Then with $v=\phi_t(s)$, $$1+C+\mu\leqslant \|y+Az^v_1 +A\sum_{i=2}^{n+1} z_i^v\|\leqslant 1+C,$$ a contradiction. This contradiction finishes the proof.

\end{proof}

\begin{rem}\upshape It follows from the proof that for any $1<p\leqslant \infty$, any ordinal $\xi$, any operator $A:X\to Y$, and any $C>\textbf{t}_{\xi,p}(A)$, $A$ satisfies $C$-$\xi$-$\ell_p$ (resp. $c_0$ if $p=\infty$) upper tree estimates.

It also follows from the proof and Proposition \ref{quant} that for each $1<p\leqslant \infty$, there exists a constant $C_p\geqslant 1$ such that if  $\|A\|<1$ and $A$ satisfies $C$-$\xi$-$\ell_p$ (resp. $c_0$ if $p=\infty$) upper tree estimates for some $C<1$, then there exists an equivalent norm $|\cdot|$ on $Y$ such that $$\frac{1}{2}B_Y \subset B_Y^{|\cdot|}\subset 2B_Y$$ and $\textbf{t}_{\xi,p}(A:X\to (Y, |\cdot|))\leqslant C_p$.

\label{renorm remark}
\end{rem}

\section{$C(K)$ spaces and injective tensor products}

\subsection{$C(K)$ spaces}

Our first result of this section is the optimal renorming theorem regarding asymptotic smoothness for $C(K)$ spaces. Recall that for a compact, Hausdorff space $K$ and a closed subset $L$ of $K$, $L'$ denotes the subset of $L$ consisting of those members of $L$ which are not isolated relative to $L$. We define the transfinite Cantor-Bendixson derivatives by $$K^0=K,$$ $$K^{\xi+1}=(K^\xi)',$$ and if $\xi$ is a limit ordinal, $$K^\xi=\bigcap_{\zeta<\xi}K^\zeta.$$  We say $K$ is \emph{scattered} provided that there exists an ordinal $\xi$ such that $K^\xi=\varnothing$, and in this case we let $CB(K)$ be the minimum such $K$.  Of course, we are implicitly assuming $K$ is non-empty, so that if $K$ is scattered, compactness yields that $CB(K)$ is a successor ordinal. Regarding the renorming of $C(K)$, we have the following.  The proof is a transfinite analogue of a result of Lancien from \cite{L}, wherein the $\xi=0$ case of the following result is proved.

\begin{theorem} For any ordinal $\xi$ and any compact, Hausdorff space $K$, the following are equivalent. \begin{enumerate}[(i)]\item $CB(K)< \omega^{\xi+1}$, \item $Sz(C(K))\leqslant \omega^{\xi+1}$, \item $C(K)$ is $\xi$-asymptotically uniformly smoothable. \item $C(K)$ is $\xi$-asymptotically uniformly flattenable. \end{enumerate}

\label{ck}

\end{theorem}

\begin{proof} The equivalence of $(i)$ and $(ii)$ follows from \cite{C}, while the equivalence of $(ii)$ and $(iii)$ follows from \cite{CD}. Of course, $(iv)\Rightarrow (iii)$.  Thus it suffices to show that $(i)\Rightarrow (iv)$.  Assume that $CB(K)<\omega^{\xi+1}$, from which we deduce the existence of some $n\in \nn$ such that $K^{\omega^\xi n}=\varnothing$. For a measure $\mu \in C(K)^*$ and a Borel subset $F$ of $K$, let $\mu|_F$ be the measure given by $\mu|_F(E)=\mu(E\cap F)$. Now define $|\cdot|$ on $C(K)^*$ by $$|\mu|=\sum_{i=1}^n 2^{-i}\|\mu|_{K^{\omega^\xi (i-1)}\setminus K^{\omega^\xi i}}\|.$$ It is clear that this is an equivalent norm on $C(K)^*$ which is weak$^*$ lower semi-continuous.  Therefore this is the dual norm to an equivalent norm $|\cdot|$ on $C(K)$. It remains to show that $(C(K), |\cdot|)$ is $\xi$-AUF. We will show that $\varrho_\xi(1/2, (C(K), |\cdot|))=0$.  To obtain a contradiction, assume that $D$ is a weak neighborhood basis at $0$ in $C(K)$ and that $f\in B_{C(K)}^{|\cdot|}$, $(f_t)_{t\in \Gamma_{\xi,1}.D}\subset \frac{1}{2}B_{C(K)}^{|\cdot|}$ is normally weakly null and $\epsilon>0$ is such that $$1+\epsilon<\inf_{t\in MAX(\Gamma_{\xi,1}.D)} |f+\sum_{\varnothing<s\leqslant t}\mathbb{P}_{\xi, 1}(s)f_s|.$$   For each $1\leqslant i\leqslant n$, let $$S_i=\{2^i \ee \delta_\varpi: |\ee|=1, \varpi\in K^{\omega^\xi(i-1)}\}$$ and let $S=\cup_{i=1}^n S_i$. Note that $S$ is weak$^*$-compact and  $\overline{\text{abs\ co}}^{\text{weak}^*}(S)= B_{C(K)^*}^{|\cdot|}$.  From this it follows that for every $t\in MAX(\Gamma_{\xi,1}.D)$, we may fix $f^*_t\in S$ such that $$|f+\sum_{\varnothing<s\leqslant t} \mathbb{P}_{\xi,1}(s)f_s| = \text{Re\ }f^*_t\bigl(f+\sum_{\varnothing <s\leqslant t}\mathbb{P}_{\xi,1}(s)f_s\bigr).$$   We now apply Lemma \ref{stabilize} to the function $F(s,t)=\text{Re\ }f^*_t(f+f_s)$ to deduce the existence of a map $d:\Gamma_{\xi,1}.D\to \Gamma_{\xi,1}.D$ with extension $e:MAX(\Gamma_{\xi,1}.D)\to MAX(\Gamma_{\xi,1}.D)$ such that for every $(s,t)\in \Pi(\Gamma_{\xi,1}.D)$, $\text{Re\ }f^*_{e(t)}(f+f_{d(s)})> 1+\epsilon$ and such that $(f_{d(s)})_{s\in \Gamma_{\xi,1}.D}$ is weakly null.  By applying Lemma \ref{stabilize} and another map with extension and relabeling, we may assume there exist $1\leqslant i\leqslant n$  and $\alpha\leqslant 1$ such that for every $t\in MAX(\Gamma_{\xi,1}.D)$, $f^*_t\in S_i$ and $|\alpha- \text{Re\ }f^*_t(f)|\leqslant \epsilon/2$.  

We now make the following claim: For any $0\leqslant \zeta<\omega^\xi$ and any $t\in \Gamma_{\xi,1}^\zeta.D$,  there exist $\ee$ with $|\ee|=1$ and $\varpi \in K^{\omega^\xi(i-1)+\zeta}$ such that $|\alpha - \text{Re\ }2^i \ee \delta_\varpi(f)|\leqslant \epsilon/2$ and for any $\varnothing<s\leqslant t$, $\text{Re\ }2^i\ee \delta_\varpi(f+f_s)\geqslant 1+\epsilon$. The base case follows immediately from the previous paragraph. For the limit ordinal case, we assume $t\in \Gamma_{\xi,1}^\zeta.D= \cap_{\mu<\zeta} \Gamma_{\xi,1}^\mu.D$. For every $\mu<\zeta$, we fix $\ee_\mu$ and $\varpi_\mu$ as in the claim. Then any weak$^*$-limit of a subnet of $(2^i \ee_\mu \delta_{\varpi_\mu})_{\mu<\zeta}$ must be of the form $2^i \ee \delta_\varpi$, where $\ee$ and $\varpi$ are the ones we seek. Assume $\zeta+1<\omega^\xi$ and the conclusion holds for $\zeta$. Then for $t\in \Gamma_{\xi, 1}^{\zeta+1}.D$, we may fix some $\gamma$ such that for every $u\in D$, $t\smallfrown (\gamma, u)\in \Gamma_{\xi,1}^\zeta.D$.  By the inductive hypothesis, for every $u\in D$, we may fix $\ee_u$ with $|\ee_u|=1$ and $\varpi_u\in K^{\omega^\xi(i-1)+\zeta}$ such that $|\alpha-\text{Re\ }2^i \ee_u \delta_{\varpi_u}(f)|\leqslant \epsilon/2$, and for every $\varnothing<s\leqslant t\smallfrown (\gamma, u)$, $\text{Re\ }2^i \ee_u \delta_{\varpi_u}(f+f_s)\geqslant 1+\epsilon$. Let $f^*=2^i \ee \delta_\varpi$ be a weak$^*$-limit of a subnet $(2^i \ee_u \delta_{\varpi_u})_{u\in D'}$ of $(2^i \ee_u \delta_{\varpi_u})_{u\in D}$. We only need to show that $\varpi\in K^{\omega^\xi(i-1)+\zeta+1}$. But $2^i \ee_u \delta_{\varpi_u}\underset{\text{weak}^*, u\in D'}{\to} 2^i \ee \delta_\varpi$ means $\varpi_u\to \varpi$ in $K$. Since $\varpi_u\in K^{\omega^\xi(i-1)+\zeta}$, we only need to show that $\varpi \neq \varpi_u$ for all $u$ in a cofinal subset of $D'$. This follows immediately from the fact that since $$\underset{u\in D'}{\lim\inf} \text{Re\ } 2^i \ee_u \delta_{\varpi_u}(f_{t\smallfrown (\gamma, u)}) \geqslant \underset{u\in D'}{\lim\inf} \text{Re\ }2^i \ee_u \delta_{\varpi_u}(f+f_{t\smallfrown (\gamma, u)}) - |2^i \ee_u \delta_{\varpi_u}||f| \geqslant \epsilon,$$  it must be that $$\underset{u\in D'}{\lim\inf} |\delta_{\varpi_u}-\delta_\varpi| \geqslant \underset{u\in D'}{\lim\inf} \text{Re\ }(\delta_{\varpi_u}-\delta_\varpi)(\ee_u f_{t\smallfrown (\gamma, u)}) \geqslant \epsilon/2^i>0.$$  This finishes the claim.

We now claim that there exist a scalar $\ee$ with $|\ee|=1$ and $\varpi\in K^{\omega^\xi i}$ such that $|\alpha - \text{Re\ }2^i \ee \delta_\varpi(f)|\leqslant \epsilon/2$. Indeed, if $\xi=0$, this can be deduced as in the successor case from the previous paragraph. If $\xi>0$, this can be deduced as in the limit ordinal case in the previous paragraph.  But since $2(2^i \ee \delta_\varpi)\in S_{i+1}$, it follows that $|2^i \ee \delta_\varpi|\leqslant 1/2$, whence $\alpha \leqslant \text{Re\ }2^i \ee f^*(f)+\epsilon/2 \leqslant 1/2+ \epsilon/2$. But then for any $t\in MAX(\Gamma_{\xi,1}.D)$,   $$1+\epsilon < \text{Re\ }f^*_t(f+f_t) \leqslant \alpha +\epsilon/2+ |f^*_t| |f_t| \leqslant 1/2+\epsilon+1/2 = 1+\epsilon,$$ and this contradiction finishes the proof.

\end{proof}

\begin{rem}\upshape We remark that for metrizable $K$, Theorem \ref{ck} follows from Lemma \ref{duality}, \cite{HL}, and the classical isomorphic characterization of separable $C(K)$ spaces due to Bessaga and Pe\l czy\'{n}ski. But for non-separable $C(K)$ spaces, this result is new. 

\end{rem}

\begin{rem}\upshape If $K$ is scattered, then there exists an ordinal $\xi$ such that $K^\xi\neq \varnothing$ and $K^{\xi+1}=\varnothing$. In this case, we let $C_0(K)=\{f\in C(K): f|_{K^\xi}\equiv 0\}$. It follows from \cite{C} that if $CB(K)=\omega^\xi+1$, then $C_0(K)$ is $\xi$-asymptotically uniformly flat.  However, if $\omega^\xi+1<CB(K)$, $C_0(K)$ is not $\xi$-asymptotically uniformly flat, or even $\xi$-asymptotically uniformly smooth. Moreover, if $\omega^\xi+1\leqslant CB(K)$, $C(K)$ is not $\xi$-asymptotically uniformly smooth. Indeed, if $\varpi\in K^{\omega^\xi+1}$, one can easily construct a weak$^*$-null tree $(\mu_t)_{t\in T.D}$ in which $\mu_t=\delta_{\varpi_t}-\delta_{\varpi_{t^-}}$ for some $\varpi_t\in K$ such that $t\in T^\zeta.D$ if and only if $\varpi_t\in K^\zeta$. Then $$\sup_{t\in T.D}\|\delta_\varpi + \sum_{\varnothing<s\leqslant t}\mu_s\|=1,$$ yielding that $\delta^{\text{weak}^*}_\xi(2, C(K))=0$ whenever $\omega^\xi+1\leqslant CB(K)$. Replacing $\mu_t$ with $\mu_t|_{C_0(K)}$, one can see that $\delta^{\text{weak}^*}_\xi(2, C_0(K))=0$. We will say more about this in the final section.

\end{rem}

We now move on to injective tensor products.  We recall that for Banach spaces $X,Y$, $X\hat{\otimes}_\ee Y$ is the norm closure in $\mathfrak{L}(X^*,Y)$ of the operators of the form $u=\sum_{i=1}^n x_i\otimes y_i$. We recall that the notation $u=\sum_{i=1}^nx_i\otimes y_i$ means that $$u(x^*)= \sum_{i=1}^n x^*(x_i)y_i.$$  

If $A_0:X_0\to Y_0$, $A_1:X_1\to Y_1$ are operators, there exists a unique bounded, linear extension of the map $x_0\otimes x_1 \mapsto A_0 x_0\otimes A_1 x_1$ from $X_0\hat{\otimes}_\ee X_1$ into $Y_0\hat{\otimes}_\ee Y_1$. We denote this extension by $A_0\otimes A_1$.  If $A_0=I_{X_0}$ and $A_1=I_{X_1}$, then $I_{X_0}\otimes I_{X_1}= I_{X_0\hat{\otimes}_\ee X_1}$.   

We say a property $P$ which an operator may or may not possess \emph{passes to injective tensor products of operators} if $A_0\otimes A_1$ has $P$ whenever $A_0, A_1$ have $P$. We say a property $P$ which a Banach space may or may not possess \emph{passes to injective tensor products of Banach spaces} if $X_0\hat{\otimes}_\ee X_1$ has $P$ whenever $X_0, X_1$ have $P$.

Our main result in this direction is the following. 

\begin{theorem} For any $1<p<\infty$ and any ordinal $\xi$, the following properties pass to injective tensor products of Banach spaces and operators. \begin{enumerate}[(i)]\item $\xi$-asymptotic uniform smoothness. \item $\xi$-$p$-asymptotic uniform smoothness. \item $\xi$-asymptotic uniform flatness. \item $\xi$-asymptotic uniform smoothability. \item $\xi$-$p$-asymptotic uniform smoothability. \item $\xi$-asymptotic flattenability. \item $\textbf{\emph{p}}_\xi(\cdot)\leqslant p$. \end{enumerate}

\label{tensor}

\end{theorem}

It is clear that if either $A_0=0$ or $A_1=0$, $A_0\otimes A_1$ has each of the seven properties in Theorem \ref{tensor}. Thus the non-trivial case is when $A_0, A_1$ are non-zero. Theorem \ref{tensor} will follow immediately from the following lemma, using Lemma \ref{duality}, Theorem \ref{renorming}, Theorem \ref{Szlen}, and standard arguments regarding Young duality.

\begin{lemma} Let $A_0:X_0\to Y_0$, $A_1:X_1\to Y_1$ be non-zero operators and let $\xi$ be an ordinal. Let $R=\max\{ \|A_0\|, \|A_1\|\}$.  Define $$\delta(\tau)= \min \{\delta^{\text{weak}^*}_\xi(\tau, A_0), \delta^{\text{weak}^*}_\xi(\tau, A_1)\}.$$  Then for any $0<\sigma, \tau$, if $\delta(\tau)\geqslant \sigma \tau$, $\varrho_\xi(\sigma/8R, A_0\otimes  A_1)\leqslant \sigma \tau$.

\label{ks}

\end{lemma}

\begin{proof}Seeking a contradiction, assume that $\mu>0$, $u\in B_{Y_0\hat{\otimes}_\ee Y_1}$,   $(u_t)_{t\in \Gamma_{\xi,1}.D}\subset \frac{\sigma}{8R} B_{X_0\hat{\otimes}_\ee X_1}$ are such that $$\inf_{t\in MAX(\Gamma_{\xi,1}.D)} \|u+A_0\otimes A_1\sum_{\varnothing<s\leqslant t}\mathbb{P}_{\xi,1}(s) u_s\| > 1+\sigma \tau+\mu.$$   Now fix $\delta>0$ such that $\mu>3\delta+ \sigma \delta/4R$.

For each $t\in MAX(\Gamma_{\xi, 1}.D)$, fix $y^*_{0, t}\in B_{Y^*_0}$, $y^*_{1, t}\in B_{Y^*_1}$ such that $$\text{Re\ }y^*_{0, t}\otimes y^*_{1,t}\bigl(u+A_0\otimes A_1\sum_{\varnothing<s\leqslant t}\mathbb{P}_{\xi,1}(s) u_s\bigr)= \|u+A_0\otimes A_1\sum_{\varnothing<s\leqslant t}\mathbb{P}_{\xi,1}(s) u_s\|.$$  Define $h:\Pi(\Gamma_{\xi,1}.D)\to \rr$ by $h(s,t)= \text{Re\ }y^*_{0,t}\otimes y^*_{1,t} (u+A_0\otimes A_1 (u_s))$ and note that by Lemma \ref{stabilize}, there exists a monotone map $d:\Gamma_{\xi, 1}.D\to \Gamma_{\xi,1}.D$ with extension $e$ such that for each $(s,t)\in \Pi(\Gamma_{\xi,1}.D)$, $\text{Re\ }y^*_{0, e(t)}\otimes y^*_{1, e(t)}(u+A_0\otimes A_1 (u_{d(s)}))\geqslant 1+ \sigma \tau+\mu$ and such that $(u_{d(s)})_{s\in \Gamma_{\xi, 1}.D}$ is weakly null.    Now define $\chi:\Pi(\Gamma_{\xi, 1}.D)\to \rr^2$ by  $$\chi(s,t)= ( \text{Re\ }y^*_{0, e(t)}\otimes y^*_{1, e(t)}(u), y^*_{0, e(t)}\otimes y^*_{1, e(t)}(A_0\otimes A_1   u_{d(s)})).$$   By Lemma \ref{stabilize}, we may fix another monotone map $d':\Gamma_{\xi, 1}.D\to \Gamma_{\xi,1}.D$ with extension $e'$, $\alpha, \beta\in \rr$ such that $(u_{d\circ d'(s)})_{s\in \Gamma_{\xi,1}.D}$ is still weakly null, for every $t\in MAX(\Gamma_{\xi,1}.D)$, $$|\alpha- \text{Re\ }y^*_{0, e\circ e'(t)}\otimes y^*_{1, e\circ e'(t)}(u)|\leqslant \delta,$$ and for every $(s,t)\in \Pi(\Gamma_{\xi,1}.D)$, $$|\sigma\beta/8R- \text{Re\ }y^*_{0, e\circ e'(t)} \otimes y^*_{1, e\circ e'(t)}(A_0\otimes A_1(u_{d\circ d'(s)}))|\leqslant \delta $$ and $$|\beta- \text{Re\ }y^*_{0, e\circ e'(t)} \otimes y^*_{1, e\circ e'(t)}(A_0\otimes A_1( \frac{8R}{\sigma}u_{d\circ d'(s)}))|\leqslant \delta.$$   Then for any $t\in MAX(\Gamma_{\xi,1}.D)$,  $$1+\sigma \tau+\mu \leqslant \text{Re\ }y^*_{0, e\circ e'(t)}\otimes y^*_{1, e\circ e'(t)} (u+A_0\otimes A_1(u_{d\circ d'(t)})) \leqslant \alpha+\sigma \beta/8R+2\delta.$$  Since $\alpha\leqslant 1+\delta$, we deduce that $$1+\sigma \tau+\mu \leqslant 1+3\delta +\sigma \beta/8R,$$ and $$\tau \leqslant \tau + \frac{\mu}{\sigma}-\frac{3\delta}{\sigma}- \frac{\delta}{4R} \leqslant \frac{\beta-2\delta}{8R}.$$

\begin{proposition} $\alpha\leqslant 1+\delta-\sigma\bigl(\frac{\beta-2\delta}{8R}\bigr)$. 

\label{lame claim}
\end{proposition}

The proof of this claim is somewhat technical, so we relegate the proof to the end of the final section and proceed with the proof of Proposition \ref{ks}.   We deduce that \begin{align*} 1+\sigma \tau+\mu & \leqslant \alpha+ \sigma \beta/8R +2\delta  \leqslant 1+\delta -\sigma\bigl(\frac{\beta-2\delta}{8R}\bigr) + \sigma \beta /8R +2\delta \\ & = 1+3\delta +\sigma \delta/4R<1+\mu <1+\sigma \tau+\mu,\end{align*} and this contradiction finishes the proof.

\end{proof}

\begin{rem}\upshape
Of course, we have a partial converse to Theorem \ref{tensor}.  If $A_0\otimes A_1$ has any of the seven properties from Theorem \ref{tensor}, then either $A_0=0$, $A_1=0$, or $A_0, A_1$ each have that property.

\end{rem}

\section{Ideals}

In this section, we let $\textbf{Ban}$ denote the class of all Banach spaces over $\mathbb{K}$. We let $\mathfrak{L}$ denote the class of all operators between Banach spaces and for $X,Y\in \textbf{Ban}$, we let $\mathfrak{L}(X,Y)$ denote the set of operators from $X$ into $Y$. For $\mathfrak{I}\subset \mathfrak{L}$ and $X,Y\in \textbf{Ban}$, we let $\mathfrak{I}(X,Y)=\mathfrak{I}\cap \mathfrak{L}(X,Y)$. We recall that a class $\mathfrak{I}$ is called an ideal if \begin{enumerate}[(i)]\item for any $W,X,Y,Z\in \textbf{Ban}$, any $C\in \mathfrak{L}(W,X)$, $B\in \mathfrak{I}(X,Y)$, and $A\in \mathfrak{L}(Y,Z)$, $ABC\in \mathfrak{I}$, \item $I_\mathbb{K}\in \mathfrak{I}$, \item for each $X,Y\in \textbf{Ban}$, $\mathfrak{I}(X,Y)$ is a vector subspace of $\mathfrak{L}(X,Y)$. \end{enumerate}

We recall that an ideal $\mathfrak{I}$ is said to be  \emph{closed} provided that for any $X,Y\in \textbf{Ban}$, $\mathfrak{I}(X,Y)$ is closed in $\mathfrak{L}(X,Y)$ with its norm topology. 

We say a class $\mathfrak{I}$ is a \emph{right ideal} provided that items (ii) and (iii) above hold for $\mathfrak{I}$, but item (i) is replaced by the property that for any $W,X,Y\in \textbf{Ban}$, any $B\in \mathfrak{L}(W,X)$, and any $A\in \mathfrak{I}(X,Y)$, $AB\in \mathfrak{I}$.

If $\mathfrak{I}$ is an ideal and $\iota$ assigns to each member of $\mathfrak{I}$ a non-negative real value, then we say $\iota$ is an \emph{ideal norm} provided that \begin{enumerate}[(i)]\item for each $X,Y\in \textbf{Ban}$, $\iota$ is a norm on $\mathfrak{I}(X,Y)$, \item for any $W,X,Y,Z\in \textbf{Ban}$ and any $C\in \mathfrak{L}(W,X)$, $B\in \mathfrak{I}(X,Y)$, $A\in \mathfrak{I}(Y,Z)$, $\iota(ABC)\leqslant \|A\|\iota(B)\|C\|$, \item for any $X,Y\in \textbf{Ban}$, any $x\in X$, and any $y\in Y$, $\iota(x\otimes y)=\|x\|\|y\|$. \end{enumerate} We similarly defined a \emph{right ideal norm} on a right ideal $\mathfrak{I}$ by replacing item (ii) above with the property that for any $W,X,Y\in \textbf{Ban}$, any $B\in \mathfrak{L}(W,X)$, and $A\in \mathfrak{I}(X,Y)$, $\iota(AB)\leqslant \iota(A)\|B\|$. 

If $\mathfrak{I}$ is an ideal and $\iota$ is an ideal norm on $\mathfrak{I}$, we say $(\mathfrak{I}, \iota)$ is a \emph{Banach ideal} provided that for every $X,Y\in \textbf{Ban}$, $(\mathfrak{I}(X,Y), \iota)$ is a Banach space. A \emph{right Banach ideal} is defined similarly.

For $1<p<\infty$ and an ordinal $\xi$ and an operator $A:X\to Y$, we let $T_{\xi,p}(A)$ be the infimum of those $C$ such that $A$ satisfies $C$-$\xi$-$\ell_p$ upper tree estimates, where $T_{\xi,p}(A)=\infty$ if there is no such $C$. We let $T_{\xi,\infty}(A)$ be the infimum of those $C$ such that $A$ satisfies $C$-$\xi$-$c_0$ upper tree estimates. For $1<p\leqslant \infty$, we let $\mathfrak{t}_{\xi,p}(A)=\|A\|+T_{\xi,p}(A)$. We let $\mathfrak{T}_{\xi,p}$ denote the class of operators $A$ for which $\mathfrak{t}_{\xi,p}(A)<\infty$. We note that by Theorem \ref{renorming}, $\mathfrak{T}_{\xi,p}$ is the class of all $\xi$-$p$ asymptotically uniformly smoothable operators, and $\mathfrak{T}_{\xi,\infty}$ is the class of $\xi$-asymptotically uniformly flattenable operators. 

We need the following observation. 

\begin{proposition} Let $I$ be a non-empty set, $\xi$ an ordinal, $1<p<\infty$, and for each $i\in I$, let $A_i:X_i\to Y_i$ be an operator. Assume that $\sup_{i\in I} \|A_i\|<\infty$. Define $A:X:=(\oplus_{i\in I}X_i)_{\ell_p(I)}\to Y:=(\oplus_{i\in I}Y_i)_{\ell_p(I)}$ by $A|_{X_i}=A_i$. Then $$\textbf{\emph{t}}_{\xi,p}(A) \leqslant \sup_{i\in I} \max\{\textbf{\emph{t}}_{\xi,p}(A), \|A_i\|\}.$$

The analogous result holds when $p=\infty$ if we replace $\ell_p(I)$ with $c_0(I)$.

\end{proposition}

\begin{proof} We prove the result in the case $1<p<\infty$, with the $p=\infty$ case requiring only notational changes. Let $C=\sup_{i\in I} \max\{ \textbf{t}_{\xi, p}(A), \|A_i\|\}$. Fix $y=(y_i)_{i\in I}\in Y$ such that $J:=\{i\in I: y_i\neq 0\}$ is finite.  Fix $\sigma>0$ and a weakly null $(x_t)_{t\in \Gamma_{\xi,1}.D}\subset \sigma B_X$.  To obtain a contradiction, assume $$\inf_{t\in MAX(\Gamma_{\xi,1}.D} \|y+A\sum_{\varnothing<s\leqslant t}\mathbb{P}_{\xi,1}(s) x_s\|^p > \|y\|^p + (C+\delta)^p( \sigma+\delta)^p$$ for some $\delta>0$.   As usual, we may apply Lemma \ref{stabilize}, relabel, and assume that for each $t\in MAX(\Gamma_{\xi,1}.D)$, there exists $y^*_t\in B_{Y^*}$ such that for each $\varnothing<s\leqslant t$, $$\text{Re\ }y^*_t(y+Ax_s) \geqslant (\|y\|^p+ (C+\delta)^p (\sigma+\delta)^p)^{1/p}.$$

 Fix a finite subset $S$ of $\ell_p^{1+|J|}$ such that \begin{enumerate}[(i)]\item for each $(a,a_j)_{j\in J}\in S$, $a, a_j> 0$, \item for each $(a, a_j)_{j\in J}\in S$, $a^p+\sum_{j\in J}a_j^p \leqslant (\sigma+\delta)^p$, \item for any $(b, b_j)_{j\in J}\in \sigma B_{\ell_p^{1+|J|}}$, there exists $(a, a_j)_{j\in J}\in S$ such that $|b|\leqslant a$ and $|b_j|\leqslant a_j$ for each $j\in J$. \end{enumerate} For each $t\in \Gamma_{\xi,1}.D$, there exists $(a^t, a_j^t)_{j\in J}\in S$ such that $\|P_{I\setminus J}x_t\|\leqslant a^t$ and $\|x_{t,j}\|\leqslant a_j^t$ for each $j\in J$. We may use Lemma \ref{stabilize} and relabel once again to assume there exists $(a, a_j)_{j\in J}\in S$ such that $(a^t, a_j^t)_{j\in J}=(a, a_j)_{j\in J}$ for all $t\in \Gamma_{\xi,1}.D$.  Thus we arrive at a weakly null collection $(x_t)_{t\in \Gamma_{\xi,1}.D}\subset B_X$ such that $$\|y+A\sum_{\varnothing<s\leqslant t}\mathbb{P}_{\xi,1}(s)x_s\|^p \geqslant \|y\|^p+(C+\delta)^p (\sigma+\delta)^p$$ for every $t\in MAX(\Gamma_{\xi,1}.D)$, and $$\|P_{I\setminus J}x_t\|\leqslant a,$$ $$\|x_{t,j}\|\leqslant a_j$$ for every $j\in J$ and $t\in \Gamma_{\xi,1}.D$.

Now for each $j\in J$, let $\mathcal{E}_j$ denote the set of those $t\in MAX(\Gamma_{\xi,1}.D)$ such that  $$\|y_i+A_i \sum_{\varnothing<s\leqslant t}\mathbb{P}_{\xi,1}(s) x_{s,j}\|^p \geqslant \|y_i\|^p+(C+\delta)^p a_j^p.$$  Then for each $j\in J$, $\mathcal{E}_j$ fails to be cofinal, whence there exists $t\in MAX(\Gamma_{\xi,1}.D)\setminus \cup_{j\in J}\mathcal{E}_j$.  From this we deduce that \begin{align*} \|y+A\sum_{\varnothing<s\leqslant t}\mathbb{P}_{\xi, 1}(s) x_s \|^p  & \leqslant \|P_{I\setminus J}A\sum_{\varnothing<s\leqslant t}\mathbb{P}_{\xi, 1}(s) x_s\|^p + \sum_{j\in J} \|y_j+A_j \sum_{\varnothing<s\leqslant t}\mathbb{P}_{\xi,1}(s) x_{s,j}\|^p \\ & < C^p a^p + \sum_{j\in J} [\|y_j\|^p +(C+\delta)^p a_j^p] \leqslant  \sum_{j\in J}\|y_j\|^p +(C+\delta)^p [a^p+\sum_{j\in J} a_j^p] \\ & \leqslant \|y\|^p+ (C+\delta)^p (\sigma+\delta)^p. \end{align*} This is the contradiction we sought.

\end{proof}

\begin{theorem} For every ordinal $\xi$ and every $1<p\leqslant \infty$, $(\mathfrak{T}_{\xi,p}, \mathfrak{t}_{\xi,p})$ is a Banach ideal. 

\label{ideal1}

\end{theorem}

\begin{proof} It is quite clear that $\mathfrak{t}_{\xi,p}$ is positive homogeneous, takes finite values on $\mathfrak{T}_{\xi,p}$, and $\mathfrak{t}_{\xi,p}(A)=0$ if and only if $A=0$. Fix Banach spaces $X,Y$ and fix a weakly null collection $(x_t)_{t\in \Gamma_{\xi,\infty}.D}\subset B_X$, $A,B\in \mathfrak{T}_{\xi,p}(X,Y)$, $\alpha>T_{\xi,p}(A)$, $\beta> T_{\xi,p}(B)$.   Then \begin{align*} \bigcup_{n=1}^\infty &  \Bigl\{t\in MAX(\Lambda_{\xi, \infty, n}.D): (\forall (a_i)_{i=1}^n\in \mathbb{K}^n)\Bigl( \|(A+B)\sum_{i=1}^n a_i z_i^t\|\leqslant (\alpha+\beta)\|(a_i)_{i=1}^n\|_{\ell_p^n}\Bigr)\Bigr\} \\ & \supset \Biggl[\bigcup_{n=1}^\infty   \Bigl\{t\in MAX(\Lambda_{\xi, \infty, n}.D): (\forall (a_i)_{i=1}^n\in \mathbb{K}^n)\Bigl( \|A\sum_{i=1}^n a_i z_i^t\|\leqslant \alpha\|(a_i)_{i=1}^n\|_{\ell_p^n}\Bigr)\Bigr\} \Biggr] \\ & \cap \Biggl[\bigcup_{n=1}^\infty   \Bigl\{t\in MAX(\Lambda_{\xi, \infty, n}.D): (\forall (a_i)_{i=1}^n\in \mathbb{K}^n)\Bigl( \|B\sum_{i=1}^n a_i z_i^t\|\leqslant \beta\|(a_i)_{i=1}^n\|_{\ell_p^n}\Bigr)\Bigr\} \Biggr] .\end{align*} Since the two latter sets are big, supersets of big sets are big, and the intersection of two big sets is big, we deduce that $T_{\xi,p}(A+B)\leqslant T_{\xi,p}(A)+T_{\xi,p}(B)$ and $\mathfrak{t}_{\xi,p}(A+B)\leqslant \mathfrak{t}_{\xi,p}(A)+\mathfrak{t}_{\xi,p}(B)$.  Thus $(\mathfrak{T}_{\xi,p}(X,Y), \mathfrak{t}_{\xi,p})$ is a normed space.

It is clear that $T_{\xi,p}(A)=0$ for any compact $A$, since weakly null collections are sent to norm null collections by a compact operator. From this we deduce that $\mathfrak{t}_{\xi,p}(x\otimes y)=\|x\|\|y\|$ for any Banach spaces $X,Y$, any $x\in X$, and any $y\in Y$.  Fix $W,X,Y,Z\in \textbf{Ban}$, $C\in \mathfrak{L}(W,X)$, $B\in \mathfrak{T}_{\xi,p}(X,Y)$, $A\in\mathfrak{L}(Y,Z)$ with $\|A\|=\|C\|=1$.   Fix $(w_t)_{t\in \Gamma_{\xi, \infty}.D}\subset B_W$ weakly null and  $\beta>T_{\xi,p}(B)$.   Then \begin{align*} \bigcup_{n=1}^\infty  & \Bigl\{t\in MAX(\Lambda_{\xi, \infty, n}.D): (\forall (a_i)_{i=1}^n\in \mathbb{K}^n)\Bigl( \|ABC\sum_{i=1}^n a_i z_i^t\|\leqslant \beta\|(a_i)_{i=1}^n\|_{\ell_p^n}\Bigr)\Bigr\} \\ & \supset \bigcup_{n=1}^\infty   \Bigl\{t\in MAX(\Lambda_{\xi, \infty, n}.D): (\forall (a_i)_{i=1}^n\in \mathbb{K}^n)\Bigl( \|B\sum_{i=1}^n a_i Cz_i^t\|\leqslant \beta\|(a_i)_{i=1}^n\|_{\ell_p^n}\Bigr)\Bigr\}.  \end{align*} Since $(Cw_t)_{t\in \Gamma_{\xi, \infty}.D}\subset B_X$ is weakly null, the latter set is big, and so is the former. By homogeneity, $T_{\xi,p}(ABC)\leqslant \|A\|T_{\xi,p}(B)\|C\|$ and $\mathfrak{t}_{\xi,p}(ABC)\leqslant \|A\|\mathfrak{t}_{\xi,p}(B)\|C\|$ for any $C\in \mathfrak{L}(W,X)$, $B\in \mathfrak{T}_{\xi,p}(X,Y)$, and $C\in \mathfrak{L}(Y,Z)$. Thus $\mathfrak{T}_{\xi,p}$ is an ideal and $\mathfrak{t}_{\xi,p}$ is an ideal norm on $\mathfrak{T}_{\xi,p}$.

We last fix $X,Y\in \textbf{Ban}$ and show that $(\mathfrak{T}_{\xi,p}(X,Y), \mathfrak{t}_{\xi,p})$ is complete. Fix a $\mathfrak{t}_{\xi,p}$-Cauchy sequence $(A_n)_{n=1}^\infty$. This sequence is norm Cauchy and has a norm limit, say $A$.  To obtain a contradiction, assume $\underset{n}{\lim\sup} \mathfrak{t}_{\xi,p}(A-A_n)>\ee$ for some $0<\ee<1$. From this it follows that $\underset{n}{\lim\sup} T_{\xi,p}(A-A_n)>\ee$.  Fix positive numbers $(\ee_n)_{n=1}^\infty$ such that $4C_p\sum_{n=1}^\infty \ee_n<\ee$, where $C_p\geqslant 1$ is the constant from Remark \ref{renorm remark}. By passing to a subsequence, we may assume $T_{\xi,p}(A-A_n)>\ee$ for all $n\in \nn$, while $\mathfrak{t}_{\xi,p}(A_{n+1}-A_n)<\ee_n^{-2}$ for all $n\in \nn$.  By Remark \ref{renorm remark}, for each $n\in \nn$, there exists a norm $|\cdot|_n$ on $Y$ such that $\frac{1}{2}B_Y \subset B_Y^{|\cdot|_n}\subset 2 B_Y$ and $\textbf{t}_{\xi,p}(\ee_n^{-2}(A_{n+1}-A_n):X\to (Y, |\cdot|_n))\leqslant C_p$.  For each $n\in \nn$, let $X_n=X$ and $Y_n= (Y, |\cdot|_n)$.   Define $S_1:X\to (\oplus_{n=1}^\infty X_n)_{\ell_p}$, $S_2:(\oplus_{n=1}^\infty X_n)_{\ell_p}\to (\oplus_{n=1}^\infty Y_n)_{\ell_p}$, $S_3: (\oplus_{n=1}^\infty Y_n)_{\ell_p}\to Y$ by $$S_1 x=(\ee_n x)_{n=1}^\infty,$$ $$S_2 ((x_n)_{n=1}^\infty)= (\ee_n^{-2}(A_{n+1}-A_n)x_n)_{n=1}^\infty,$$ $$S_3((y_n)_{n=1}^\infty)= \sum_{n=1}^\infty \ee_n y_n.$$   Then $$T_{\xi,p}(S_2)\leqslant \sup_{n\in \nn} \max\{\ee^{-2}_n \|A_{n+1}-A_n: X\to Y_n\|, \textbf{t}_{\xi,p}(\ee^{-2}_n(A_{n+1}-A_n):X\to Y_n)\} \leqslant 2C_p,$$ $\|S_1\|\leqslant \sum_{n=1}^\infty \ee_n$, $\|S_3\|\leqslant 2\sum_{n=1}^\infty \ee_n$, and $$T_{\xi,p}(S_3S_2S_1)\leqslant \|S_3\|T_{\xi,p}(S_2)\|S_1\| \leqslant 4 C_p (\sum_{n=1}^\infty \ee_n)^2<\ee.$$ But $S_3S_2S_1=A-A_1$, yielding a contradiction.

\end{proof}

For each ordinal $\xi$ and $1<p<\infty$, we let $\mathfrak{G}_{\xi,p}$ denote the class of $p$-asymptotically uniformly smooth operators. For each ordinal $\xi$, we let $\mathfrak{F}_\xi$ denote the class of asymptotically uniformly flat operators. We let $$g_{\xi,p}(A)=\sup_{\sigma>0} \varrho_\xi(\sigma, A)/\sigma^p,$$ $$G_{\xi,p}(A)=\|A\|+g_{\xi,p}(A),$$ $$\mathfrak{g}_{\xi,p}(A)=\inf\{\sigma>0: G_{\xi,p}(\sigma^{-1}A)\leqslant 1\}.$$     We let $$\sigma_\xi(A)=\inf \{\sigma>0: \varrho_\xi(1/\sigma, A)=0\}=\inf\{\sigma>0: \varrho_\xi(1, \sigma^{-1} A)=0\},$$ $$\Sigma_\xi(A)=\|A\|+\sigma_\xi(A),$$  and $$\mathfrak{f}_\xi(A)=\inf\{\sigma>0: \Sigma_\xi(\sigma^{-1} A)\leqslant 1\}.$$

\begin{theorem} For any ordinal $\xi$ and any $1<p<\infty$, $(\mathfrak{G}_{\xi,p}, \mathfrak{g}_{\xi,p})$ and $(\mathfrak{F}_\xi, \mathfrak{f}_\xi)$ are right Banach ideals.

\label{ideal2}
\end{theorem}

We first recall the following classical result. 

\begin{fact} If $L$ is a Banach space, $C\subset B_L$ is a non-empty, closed, convex, balanced set with Minkowski functional $\mathfrak{f}$, then $(\text{\emph{span}}C, \mathfrak{f})$ is a Banach space.

\label{fact}
\end{fact}

\begin{proof}[Proof of Theorem \ref{ideal2}] Fix $1<p<\infty$ and Banach spaces $X,Y$. Let $C=\{A\in \mathfrak{L}(X,Y): G_{\xi,p}(A)\leqslant 1\}$. It is clear that $G_{\xi,p}(A)\geqslant \|A\|$, from which we deduce that $C\subset B_{\mathfrak{L}(X,Y)}$. We also note that $|\varrho_\xi(\sigma, A)-\varrho_\xi(\sigma, B)|\leqslant \sigma \|A-B\|$ for any $\sigma>0$ and $A,B:X\to Y$. From this it follows that if $A_n\to A$ in norm, then for any $\sigma>0$, $\varrho_\xi(\sigma, A_n)/\sigma^p\to \varrho_\xi(\sigma, A)/\sigma^p$.  Moreover, if $A_n\to A$ in norm, $g_{\xi,p}(A)\leqslant \underset{n}{\lim\sup} g_{\xi,p}(A_n)$. From this we see that if $\|A_n\|+g_{\xi,p}(A_n)\leqslant 1$ and $A_n\to A$ in norm, $$\|A\|+g_{\xi, p}(A) \leqslant \underset{n}{\lim\sup} \|A_n\|+g_{\xi,p}(A_n)\leqslant 1.$$ This shows that $C$ is closed.   Obviously $C$ is non-empty and balanced. By Corollary \ref{convexity corollary1}, for $\alpha, \beta>0$ with $\alpha+\beta=1$, $\sigma>0$, and $A,B:X\to Y$,  $$\varrho_\xi(\sigma, \alpha A+ \beta B) \leqslant \alpha \varrho_\xi(\sigma, A) + \beta \varrho_\xi(\sigma, B).$$  From this it follows that $g_{\xi,p}$ and $G_{\xi,p}$ are convex and $C$ is a convex set.  Clearly $\text{span}(C)=\mathfrak{G}_{\xi,p}(X,Y)$ and $\mathfrak{g}_{\xi,p}|_{\mathfrak{G}_{\xi,p}(X,Y)}$ is the Minkowski functional of $C$. From Fact \ref{fact}, we deduce that $(\mathfrak{G}_{\xi,p}(X,Y), \mathfrak{g}_{\xi,p})$ is a Banach space.

The same argument holds to show that $\mathfrak{F}_\xi(X,Y)$ is a Banach space, once we establish that $C'=\{A\in \mathfrak{L}(X,Y): \Sigma_\xi(A)\leqslant 1\}$ is closed and convex. If $A_n\to A$ in norm, we can  deduce that $\sigma_\xi(A)\leqslant \underset{n}{\lim\sup} \sigma_\xi(A_n)$, and deduce closedness as in the previous paragraph. For convexity, it suffices to show that $\sigma_\xi$ is a convex function on $\mathfrak{F}_\xi(X,Y)$. Fix $\alpha',\beta'>0$ with $\alpha'+\beta'=1$ and fix $A,B\in \mathfrak{F}_\xi(X,Y)$. Fix $\alpha>\sigma_\xi(A)$ and $\beta>\sigma_\xi(B)$.  Then \begin{align*} \varrho_\xi(\frac{1}{\alpha'\alpha+\beta'\beta} \alpha' A+\beta' B) & = \varrho_\xi(1, \frac{\alpha'}{\alpha'\alpha+\beta'\beta} A+ \frac{\beta'}{\alpha'\alpha+\beta'\beta}B) \\ & =\varrho_\xi(1, \frac{\alpha'\alpha}{\alpha'\alpha+\beta'\beta}\cdot \frac{A}{\alpha}+ \frac{\beta'\beta}{\alpha'\alpha+\beta'\beta}\cdot\frac{B}{\beta}) \\ & \leqslant \frac{\alpha'\alpha}{\alpha'\alpha+\beta'\beta} \varrho_\xi(1, \alpha^{-1}A) + \frac{\beta'\beta}{\alpha'\alpha+\beta'\beta}\varrho_\xi(1, \beta^{-1}B) =0.  \end{align*}

Of course, $0=\sigma_\xi(A)=g_{\xi,p}(A)$ whenever $A$ is compact, so  $\mathfrak{f}_\xi(x\otimes y)=\mathfrak{g}_{\xi,p}(x\otimes y)=\|x\|\|y\|$, since $g_{\xi,p}(x\otimes y)=\sigma_\xi(x\otimes y)=0$ for any $X,Y\in \textbf{Ban}$ and $x\in X$, $y\in Y$.

Finally, in order to show that $\mathfrak{g}_{\xi,p}(AB)\leqslant \mathfrak{g}_{\xi,p}(A)\|B\|$ (resp. $\mathfrak{f}_\xi(AB)\leqslant \mathfrak{f}_\xi(A)\|B\|$) whenever $W,X,Y\in \textbf{Ban}$, $B\in \mathfrak{L}(W,X)$, and $A\in \mathfrak{G}_{\xi,p}(X,Y)$ (resp. $\mathfrak{F}_\xi(X,Y)$), it suffices to show that $g_{\xi,p}(AB)\leqslant g_{\xi,p}(A)$ whenever  $\|B\|\leqslant 1$ (resp. $\sigma_\xi(AB)\leqslant \sigma_\xi(A)$ whenever  $\|B\|\leqslant 1$). But these follow immediately from the fact that $$\varrho_\xi(\sigma, AB)\leqslant \varrho_\xi(\sigma, A)$$ for every $\sigma>0$ whenever $\|B\|\leqslant 1$. This is because under these hypotheses, any weakly null collection of order $\omega^\xi$ in $\sigma B_W$ is sent by $B$ to a weakly null collection of order $\omega^\xi$ in $\sigma B_X$.

\end{proof}

\begin{proposition} For any ordinal $\xi$ and $1<p<\infty$, $\mathfrak{G}_{\xi,p}$ and $\mathfrak{F}_\xi$ fail to be left ideals, and $\mathfrak{G}_{\xi,p}$, $\mathfrak{F}_\xi$, $\mathfrak{T}_{\xi,p}$, $\mathfrak{T}_{\xi,\infty}$ fail to be closed. 
\label{ideal3}
\end{proposition}

We will need the following example. 

\begin{proposition} Let $\xi$ be an ordinal. \begin{enumerate}[(i)]\item There exists a weakly null collection $(f_t)_{t\in T.\nn}\subset B_{C_0([0, \omega^{\omega^\xi}])}$ such that for every $t\in T.\nn$,  $f_t\geqslant 0$ and $\|f\|=1$ for every convex combination $f$ of $(f_s: \varnothing<s\leqslant t)$. \item $\textbf{\emph{t}}_{\xi, \infty}(C_0([0, \omega^{\omega^\xi}]))=1$. \item $C([0, \omega^{\omega^\xi}])$ fails to be $\xi$-AUS.  \end{enumerate} 

\label{example}

\end{proposition}

\begin{proof}$(i)$ Let $T=\{(\gamma_i)_{i=1}^k: \omega^\xi>\gamma_1>\ldots >\gamma_k\}$. Given $t=(\gamma_i, n_i)_{i=1}^k\in T.\nn$, let $f_t$ be the indicator of the interval $$I_t:=(\omega^{\gamma_1}n_1+\ldots +\omega^{\gamma_k}n_k, \omega^{\gamma_1}n_1+\ldots \omega^{\gamma_k}(n_k+1)].$$  Note that if $t_1<\ldots <t_k$, $t_i\in T.D$, then $I_{t_1}\supset \ldots \supset I_{t_k}$, yielding that $\|f\|=1$ for any $t\in T.D$ and any $f\in \text{co}(f_s: \varnothing<s\leqslant t)$.  Moreover, an easy induction yields that for any $0\leqslant \gamma<\omega^\xi$, $$T^\gamma= \{(\gamma_i)_{i=1}^k: \omega^\xi>\gamma_1>\ldots >\gamma_k\geqslant \gamma\},$$ whence $o(T)=\omega^\xi$.  Last, if $t\smallfrown (\gamma,1)\in T$, $(I_{t\smallfrown (\gamma, n)})_{n\in \nn}$ are pairwise disjoint, and $(f_{t\smallfrown (\gamma, n)})_{n\in \nn}$ is weakly null.

 $(ii)$ Fix $f\in C_0([0, \omega^{\omega^\xi}])$ and a weakly null tree $(f_t)_{t\in T.D}\subset \sigma B_{C_0([0, \omega^{\omega^\xi}])}$. Fix $\ee>0$ and let $F=\{\varpi\in [0, \omega^{\omega^\xi}]: |f(\varpi)|\geqslant \ee\}$.  Then since $F$ is closed and $\omega^{\omega^\xi}\notin F$, there exists $\gamma<\omega^{\omega^\xi}$ such that $F\subset [0, \gamma]$. Since $Sz(C([0, \gamma]))\leqslant \omega^\xi$, if $R:C_0([0, \omega^{\omega^\xi}])\to C([0, \gamma])$ is the restriction map $Rg= g|_{[0, \gamma]}$, there exists $t\in T.D$ and $g\in \text{co}(f_s: \varnothing<s\leqslant t)$ such that $\|g|_{[0, \gamma]}\|<\ee$. The latter claim follows from \cite{CAlt}.   Then $$\|f+g\|\leqslant \max\{\|f\|+\|g|_{[0, \gamma]}\|, \|f|_{[\gamma+1, \omega^{\omega^\xi}]}\|+\|g\|\} \leqslant \max\{\|f\|+\ee, \ee+\sigma\}.$$  This yields that $\textbf{t}_{\xi, \infty}(C_0([0, \omega^{\omega^\xi}]))\leqslant 1$.  The collection from $(i)$ yields that $\textbf{t}_{\xi, \infty}(C_0([0, \omega^{\omega^\xi}]))\geqslant 1$, giving $(ii)$.

$(iii)$ If $(f_t)_{t\in T.\nn}$ is the collection from $(i)$, then for any $\sigma>0$, $1_{[0, \omega^{\omega^\xi}]}\in B_{C([0, \omega^{\omega^\xi}])}$ and $(\sigma f_t)_{t\in T.\nn}$ witness that $\varrho_\xi(\sigma, C([0, \omega^{\omega^\xi}])) \geqslant \sigma$.

\end{proof}

\begin{proof}[Proof of Proposition \ref{ideal3}] Since $C_0([0, \omega^{\omega^\xi}])$ is $\xi$-AUF by Proposition \ref{example}, while the isomorphic space $C([0, \omega^{\omega^\xi}])$ is not $\xi$-AUS, we deduce that $\mathfrak{F}_\xi$, $\mathfrak{G}_{\xi,p}$ are not left ideals. Indeed, since $\mathfrak{F}_\xi$ (resp. $\mathfrak{G}_{\xi,p}$) is a right ideal, if $\mathfrak{F}_\xi$ (resp. $\mathfrak{G}_{\xi,p}$) were a left ideal and if $A:C([0, \omega^{\omega^\xi}])\to C_0([0, \omega^{\omega^\xi}])$ is an isomorphism, $I_{C([0, \omega^{\omega^\xi}])}= A^{-1} I_{C_0([0, \omega^{\omega^\xi}])} A$ would lie in $\mathfrak{F}_\xi$ (resp. $\mathfrak{G}_{\xi,p}$).

For $0<\tau<1$, define the norm $|\cdot|_\tau$ on $\mathbb{K}\oplus C_0([0, \omega^{\omega^\xi}])$ by $$\|(a, f)\|= \max\{|a|, \tau |a|+ \|f\|\}$$ and let $X_\tau$ denote $\mathbb{K}\oplus C_0([0, \omega^{\omega^\xi}])$ with this norm. For $0<\tau<1$ and $0<\theta$, let $A_{\tau, \theta}: C_0([0, \omega^{\omega^\xi}])\to X_\tau$ by $A_{\tau, \theta}(g)= (0, \theta g)$.   By Proposition \ref{example}, since $\textbf{t}_{\xi, \infty}(C_0([0, \omega^{\omega^\xi}]))=1$, we deduce that for any $\sigma>0$, $\varrho_\xi(\sigma, A_{\theta, \tau})\leqslant \max\{0, \tau+\theta\sigma-1\}$. Indeed, if $(a, f)\in B_{X_\tau}$ and $(g_t)_{t\in \Gamma_{\xi,1}.D}\subset \sigma B_{C_0([0, \omega^{\omega^\xi}])}$ is weakly null, $$\inf \{\|(a, f+\theta g)\|: t\in \Gamma_{\xi,1}.D, g\in \text{co}(g_s: \varnothing<s\leqslant t)\} \leqslant \max\{|a|, \tau|a|+\|f\|, \tau + \theta \sigma\}.$$    Furthermore, using Proposition \ref{example}$(i)$, we deduce that $\varrho_\xi(\sigma, A_{\theta, \tau})\geqslant \max\{0, \tau+\theta \sigma-1\}$, yielding equality. In particular, $\sigma_\xi(A_{\theta, \tau})=\frac{\theta}{1-\tau}$.  Let $W_n=C_0([0, \omega^{\omega^\xi}])$ for each $n\in \nn$ and  define $B, B_k:(\oplus_{n=1}^\infty W_n )_{c_0}\to (\oplus_{n=1}^\infty X_{1-1/n})_{c_0}$ by $$B|_{W_n}= A_{1/\log_2(n+1), 1-1/n},$$ $$B_k|_{W_n}= A_{1/\log_2(n+1), 1-1/n}, \hspace{5mm} n\leqslant k,$$ $$B_k|_{W_n}=0, \hspace{5mm} n>k.$$   Then since $A_{\theta, \tau}$ is $\xi$-AUF for each $\theta, \tau$, each $B_k$ is $\xi$-AUF.  Moreover, $B_k\to B$ in norm. However, $B$ is not $\xi$-$p$-AUS for any $1<p<\infty$. Indeed, fix $1<p<\infty$, $C>0$, let $q$ be such that $1/p+1/q=1$. Fix  $n\in \nn$ such that $ C^{1/p} 2\log_2(n+1)< n^{1/q}.$    Let  $\sigma= \frac{2\log_2(n+1)}{n}$ and note that \begin{align*} \varrho_\xi(B) & \geqslant \varrho_\xi(\sigma, A_{1/\log_2(n+1), 1-1/n}) \\ & =  1-1/n + \frac{1}{\log_2(n+1)}\cdot \frac{2\log_2(n+1)}{n} -1 \\ & = \frac{1}{n}> C\sigma ^p. \end{align*}

For the non-closedness of $\mathfrak{T}_{\xi,p}$, we fix $\theta_n= 1/\log_2(n+1)$ for each $n\in \nn$.  We let $A, A_k:(\oplus_{n=1}^\infty C([0, \omega^{\omega^\xi n}]))_{c_0}\to (\oplus_{n=1}^\infty C([0, \omega^{\omega^\xi n}]))_{c_0}$ be given by $A|_{C([0, \omega^{\omega^\xi n}])}= \theta_n I_{C([0, \omega^{\omega^\xi n}])}$, $A_k|_{C([0, \omega^{\omega^\xi n}])}= \theta_n I_{C([0, \omega^{\omega^\xi}]}$ if $n\leqslant k$, and $A_k|_{C([0, \omega^{\omega^\xi n}])}=0$.   Then each $A_k$ lies in $\mathfrak{T}_{\xi, \infty}$, $A_k\to A$ in the operator norm, and $Sz(A, \theta_n) \geqslant Sz(A_n, \theta_n) \geqslant n$. But this implies that $\textbf{p}_\xi(A)=\infty$, whence $A\notin \cup_{1<p\leqslant \infty}\mathfrak{T}_{\xi,p}$ by Theorem \ref{Szlen}.

\end{proof}

\begin{rem}\upshape We now discuss distinctness of the classes.  Let $\textbf{Sz}_\xi$ denote the class of operators with Szlenk index not exceeding $\omega^\xi$. It was shown in \cite{Br} that this is a closed operator ideal. For any $1<p<\infty$,  $$\textbf{Sz}_\xi\subsetneq\mathfrak{F}_\xi \subsetneq  \mathfrak{G}_{\xi,p}\subsetneq \textbf{Sz}_{\xi+1}.$$ 

The first and last containments follow from Theorems \ref{g} and \ref{Szlen}. The identity operator of $C_0([0, \omega^{\omega^\xi}])$ shows that the first containment is proper. It was explained in \cite{CD} how a construction from \cite{C3} demonstrates that the last containment is proper. The middle containment is obvious, while a construction from \cite{CD} provided for each $1<p<r<\infty$ a Banach space whose identity lies in $\mathfrak{G}_{\xi,p}\setminus \mathfrak{T}_{\xi,r}$.  Of course, we have already shown the existence of an operator in $\mathfrak{T}_{\xi,p}\setminus \cup_{1<p} \mathfrak{G}_{\xi,p}$. This fully elucidates the relationships between all classes discussed here. 

We remark that the distinctness of all classes can be witnessed by an identity operator. However, the construction from \cite{C3} actually shows that when $\xi>0$, there is an identity operator in $\textbf{Sz}_{\xi+1}\setminus \cup_{1<p<\infty} \mathfrak{T}_{\xi,p}$, while it is known \cite{R} that no identity operator can lie in this set difference when $\xi=0$. This is because of the submultiplicativity of the $\ee$-Szlenk index of a Banach space.

\end{rem}

\section{Technical lemmata}

Our first task in this section is to explain how our definition of weakly null differs from that given in \cite{CD}, and why these notions give the same modulus $\varrho_\xi$. The same arguments apply to $\delta^{\text{weak}^*}_\xi$.   There, a collection $(x_s)_{s\in S}$ was called \emph{weakly null of order} $\omega^\xi$ if $o(S)=\omega^\xi$ and for any $s\in (\{\varnothing\}\cup S)^{\zeta+1}$, $$0\in \overline{\{x_t: t\in S^\zeta, t^-=s\}}^{\text{weak}}.$$    Every collection $(x_t)_{t\in T.D}$ which is weakly null of order $\omega^\xi$ according to our definition is weakly null by the definition of \cite{CD}. However, by the \cite{CD} definition, weakly null trees may contain many ``superfluous'' branches, which is an obstruction to the usefulness of the notions of ``cofinal'' and ``eventual'' for trees. 

Now suppose that $(x_s)_{s\in S}$ is weakly null of order $\omega^\xi$ and for each $s\in S$, let $o_S(s)=\max\{\zeta<\omega^\xi: s\in S^\zeta\}$. Fix any weak neighborhood basis $D$ at $0$ in $X$. For $t\in \Gamma_{\xi,1}.D$, let $o_{\Gamma_{\xi,1}.D}(t)=\max\{\zeta<\omega^\xi: t\in (\Gamma_{\xi,1}.D)^\zeta\}$.  Then one can recursively define a monotone map $\Theta:\Gamma_{\xi,1}.D\to S$ such that for any $t=(\zeta_i, u_i)_{i=1}^n\in \Gamma_{\xi,1}.D$, $x_{\Theta(t)}\in u_n$. Indeed, suppose $t=(\zeta, u)\in MIN(\Gamma_{\xi,1}.D)$ and let $\gamma=o_{\Gamma_{\xi,1}.D}(t)$.    Then $\gamma<\omega^\xi$, and there exists $u\in (\varnothing\cup S)^{\gamma+1}$.   We may let $\Theta(t)=v$, where $v\in S^\gamma$ is such that $v^-=u$ and $x_v\in u$. Now suppose that $t=(\zeta_i, u_i)_{i=1}^n\in \Gamma_{\xi,1}.D$ for $n>1$, $\Theta(t)$ has been defined, and $o_S(\Theta(t^-))\geqslant o_{\Gamma_{\xi,1}.D}(t^-)$. Since $\gamma:=o_{\Gamma_{\xi,1}.D}(t)<o_{\Gamma_{\xi,1}.D}(t^-)$, $\Theta(t^-)\in S^{\gamma+1}$. Then we may fix some $w\in S^\gamma$ such that $w^-=\Theta(t^-)$ and $x_w\in u_n$ and let $\Theta(t)=w$.

Thus we arrive at a weakly null collection $(x_{\Theta(T)})_{t\in \Gamma_{\xi,1}.D}$ all of whose branches are subsequences of branches of the collection $(x_t)_{t\in S}$.  This shows that a collection which is weakly null according to the definition of \cite{CD} has a subcollection which is weakly null according to our definition. From this it is easy to see why these two notions give rise to the same $\varrho_\xi$ modulus.

Our next task in this section is to prove Lemma \ref{stabilize}.

\begin{rem}\upshape We note that for $\xi\leqslant \zeta$ and any directed set $D$, $\Gamma_{\xi, 1}.D$ can be mapped into $\Gamma_{\zeta, 1}.D$. More specifically, if $(x_t)_{t\in \Gamma_{\zeta, 1}.D}$ is weakly null, and if $\xi\leqslant \zeta$, there exist a monotone map $d:\Gamma_{\xi, 1}.D\to \Gamma_{\zeta, 1}.D$ such that $(x_{d(t)})_{t\in \Gamma_{\xi,1}.D}$ is weakly null.  Furthermore, since $\Gamma_{\zeta, 1}.D$ is well-founded, this $d$ admits an extension, $e$.  To see the existence of this $d$, we prove by induction on $\gamma\geqslant 0$ that if $(x_t)_{t\in \Gamma_{\xi+\gamma,1}.D}$ is weakly null, there exists a monotone map $d:\Gamma_{\xi,1}.D\to \Gamma_{\xi+\gamma, 1}.D$ such that $(x_{d(t)})_{t\in \Gamma_{\xi,1}.D}$ is weakly null. If $\gamma=0$, we can $d$ to be the identity. If we have the result for some $\gamma$ and if $(x_t)_{t\in \Gamma_{\xi+\gamma+1,1}.D}$ is weakly null, we note that $\Gamma_{\xi+\gamma, 1}.D=\Lambda_{\xi+\gamma+1, 1, 1}.D\subset \Gamma_{\xi+\gamma+1,1}.D$.  By the inductive hypothesis, there exists $d:\Gamma_{\xi, 1}.D\to \Gamma_{\xi+\gamma, 1}.D\subset \Gamma_{\xi+\gamma+1, 1}.D$ such that $(x_{d(t)})_{t\in \Gamma_{\xi, 1}.D}$ is weakly null. Last, if $\gamma$ is a limit ordinal, we may take $d$ to be the composition of the inclusion of $\Gamma_{\xi, 1}.D$ into $\Gamma_{\xi+1, 1}.D$ together with the canonical identification of $\Gamma_{\xi+1, 1}.D$ with $(\omega^\xi+\Gamma_{\xi+1, 1}).D\subset \Gamma_{\xi+\gamma, 1}.D$.

Furthermore, for any ordinal $\xi$ and any $m,n\in \nn$ with $m\leqslant n$, there exists a monotone map, and in fact a canonical identification, of $\Gamma_{\xi, m}.D$ with $\cup_{i=1}^m \Lambda_{\xi, n, i}.D$. since $\Gamma_{\xi, n}.D$ is well-defined, this identification also admits an extension $e$. 

\label{stupid remark}
\end{rem}

The following result is an inessential modification of \cite[Proposition $3.3$]{C3}. 

\begin{proposition} Given an ordinal $\xi$, $m,n\in \nn$ with $m\leqslant n$, $1\leqslant s_1<\ldots <s_m\leqslant n$, a Banach space $X$, a directed set $D$, and a weakly null collection $(x_t)_{t\in \Gamma_{\xi, n}.D}\subset X$, there exists a monotone map $d:\Gamma_{\xi, m}.D\to \Gamma_{\xi, n}.D$ such that $d(\Lambda_{\xi, m, i}.D)\subset \Lambda_{\xi, n, s_i}.D$.

\label{stupid1}
\end{proposition}

We now move to the proof of Lemma \ref{stabilize}.  This lemma has four parts, and our strategy will be to prove by induction on $\ord\times \nn$ with lexicographical order that each of the four claims holds for a given $(\xi, n)$. For convenience, we restate the lemma.

\begin{lemma} Suppose that $\xi$ is an ordinal, $n\in \nn$, $X$ is a Banach space, and $(x_t)_{t\in \Gamma_{\xi,n}.D}$ is weakly null.  \begin{enumerate}[(i)]\item If $\mathcal{E}\subset MAX(\Gamma_{\xi,n}.D)$ is cofinal, there exists a level map $d:\Gamma_{\xi,n}.D\to \Gamma_{\xi,n}.D$ with extension $e$ such that $e(MAX(\Gamma_{\xi,n}.D))\subset \mathcal{E}$ and $(x_{d(t)})_{t\in \Gamma_{\xi,n}.D}$ is weakly null.   \item For any $k\in \nn$, if $MAX(\Gamma_{\xi,n}.D)\supset \mathcal{E}=\cup_{i=1}^k \mathcal{E}_i\in \Omega_{\xi,n}$, then there exists $1\leqslant j\leqslant k$ such that $\mathcal{E}_j\in \Omega_{\xi,n}$. \item If $F$ is a finite set and $\chi:\Pi(\Gamma_{\xi,n}.D)\to F$ is a function, then there exist a level map $d:\Gamma_{\xi,n}.D\to \Gamma_{\xi,n}.D$ with extension $e$ and $\alpha_1, \ldots, \alpha_n\in F$ such that for any $1\leqslant i\leqslant n$ and any $\Lambda_{\xi,n,i}.D\ni s\leqslant \in MAX(\Gamma_{\xi,n}.D)$, $\alpha_i=F(d(s), e(t))$, and such that $(x_{d(t)})_{t\in \Gamma_{\xi,n}.D}$ is weakly null.   \item If $h:\Pi(\Gamma_{\xi, n}.D)\to \rr$ is bounded and if $\mathcal{E}\subset MAX(\Gamma_{\xi,n}.D)$ is cofinal, then for any $\delta>0$,  there exist $a_1, \ldots, a_n\in \rr$ and a level map $d:\Gamma_{\xi,n}.D\to \Gamma_{\xi,n}.D$ with extension $e$ such that $e(MAX(\Gamma_{\xi,n}.D))\subset \mathcal{E}$, for each $1\leqslant i\leqslant n$ and each $\Lambda_{\xi,n,i}.D\ni s\leqslant t\in MAX(\Gamma_{\xi,n}.D)$, $h(d(s), e(t))\geqslant a_i-\delta$, and for any $t\in MAX(\Gamma_{\xi,n}.D)$, $\sum_{\varnothing<s\leqslant e(t)} \mathbb{P}_{\xi,n}(s)h(s, e(t)) \leqslant \delta+\sum_{i=1}^n a_i$.   \end{enumerate}

\end{lemma}

\begin{proof} Case $1$: $(\xi, n)=(0, 1)$.  $(i)$ Write $\mathcal{E}=\{(1, u): u\in D_0\}$ for a cofinal subset $D_0$ of $D$. For each $u\in D$, fix some $v_u\in D_0$ with $u\leqslant v_u$. Let $d((1, u))=e((1, u))=(1,v_u)$. Then $e(MAX(\Gamma_{0, 1}.D))\subset \mathcal{E}$ and $(x_{d(t)})_{t\in \Gamma_{0, 1}.D}$ is weakly null.

$(ii)$ Write $\mathcal{E}_i=\{(1, u): u\in D_i\}$.  Then $\cup_{i=1}^k \mathcal{E}_i=\{(1, u):u\in \cup_{i=1}^k D_i\}\in \Omega_{0, 1}$, $\cup_{i=1}^k D_i$ is cofinal in $D$, and so must one of the sets $D_1, \ldots, D_k$ be.

$(iii)$ Note that $\Pi(\Gamma_{0, 1}.D)=\{((1,u), (1,u)): u\in D\}$. For each $\alpha\in F$, let $\mathcal{E}_\alpha=\{t\in \Gamma_{0, 1}.D: \chi(t,t)=\alpha\}$.  Then by $(i)$ and $(ii)$, there exists $\alpha\in F$ such that $\mathcal{E}_\alpha\in \Omega_{0,1}$.  For this we are using the fact that $MAX(\Gamma_{0,1}.D)\in \Omega_{0,1}$. Fix $d$ and $e$ as in $(i)$ with $\mathcal{E}=\mathcal{E}_\alpha$.

$(iv)$ Since the range of $h$ is totally bounded, we may fix a finite subset $F$ of $\rr$ such that $\text{range}(h)\subset \cup_{a\in F}[a-\delta, a+\delta]$.  For each $a\in F$, let $\mathcal{E}_a=\{t\in \mathcal{E}: h(t,t)\in [a-\delta, a+\delta]\}$.  Then there exists $a\in F$ such that $\mathcal{E}_a\in \Omega_{0, 1}$. Fix $d,e$ as in $(i)$ with $\mathcal{E}$ replaced by $\mathcal{E}_a$ and note that this $d,e$ satisfies the conclusions.

Case $2$: If $\xi$ is a limit ordinal and $(\zeta+1, 1)$ holds for all $\zeta<\xi$, then $(\xi, 1)$ holds.

$(i)$ Assume $\mathcal{E}\in \Omega_{\xi, 1}$.  Then there exists a cofinal subset $M$ of $[0, \xi)$ such that, with $\Theta_{\zeta+1}= (\omega^\zeta+\Gamma_{\zeta+1, 1}).D$, $\mathcal{E}\cap \Theta_{\zeta+1}$ is cofinal in $\Theta_{\zeta+1}$ for all $\zeta\in M$. For each $\zeta\in M$,  using canonical identifications, there exists a monotone map $d_\zeta:\Theta_{\zeta+1}\to \Theta_{\zeta+1}$ with extension $e_\zeta$ such that $(x_{d_\zeta(t)})_{t\in \Theta_{\zeta+1}}$ is weakly null and $e(MAX(\Theta_{\zeta+1}))\subset \mathcal{E}$. Now for each $\eta<\xi$, fix $\zeta_\eta\in M$ such that $\eta+1\leqslant \zeta_\eta+1$. By Remark \ref{stupid remark} and canonical identifications, there exists $d_\eta':\Theta_{\eta+1}\to \Theta_{\zeta_\eta+1}$ such that $(x_{d_{\zeta_\eta}\circ d'_\eta(t)})_{t\in \Theta_{\eta+1}}$ is weakly null. Let $e'_\eta$ be any extension of $d'_\eta$ and define $d, e$ by letting $d|_{\Theta_{\eta+1}}=d_{\zeta_\eta}\circ d'_\eta$ and $e|_{MAX(\Theta_{\zeta+1}}= e_{\zeta_\eta}\circ e'_\eta$.

$(ii)$ Assume $\mathcal{E}\in \Omega_{\xi, 1}$. Let $\Theta_{\zeta+1}$ be as in the previous paragraph. There exists $M\subset [0, \xi)$ cofinal in $[0, \xi)$ such that for every $\zeta\in M$, $\Theta_{\zeta+1}\cap \mathcal{E}$ is cofinal in $\Theta_{\zeta+1}$. By the inductive hypothesis, this means that for each $\zeta\in M$, there exists $1\leqslant j_\zeta\leqslant k$ such that $\Theta_{\zeta+1}\cap \mathcal{E}_{j_k}$ is cofinal in $\Theta_{\zeta+1}$. Let $M_j=\{\zeta\in M: j=j_\zeta\}$. Then there exists $1\leqslant j\leqslant k$ such that $M_j$ is cofinal in $M$, whence $\mathcal{E}_j\in \Omega_{\xi, 1}$.

$(iii)$ Again, let $\Theta_{\zeta+1}$ be as in the two previous paragraphs. Applying the inductive hypothesis to $\chi|_{\Pi(\Theta_{\zeta+1})}$, we obtain $d_\zeta:\Theta_{\zeta+1}\to \Theta_{\zeta+1}$ with extension $e_\zeta$ and $\alpha_\zeta\in F$ such that for each $(s,t)\in \Pi(\Theta_{\zeta+1})$, $\chi(d_\zeta(s), e_\zeta(t))=\alpha_\zeta$, and such that $(x_{d_\zeta(t)})_{t\in \Theta_{\zeta+1}}$ is weakly null.  For each $\alpha\in F$, let $M_\alpha= \{\zeta<\xi: \alpha_\zeta=\alpha\}$ and fix $\alpha\in F$ such that $M_\alpha$ is cofinal in $[0, \xi)$. We now define $d'_\eta$, $e'_\eta$ and then $d,e$ as in $(i)$ of Case $2$.

$(iv)$ Fix a cofinal subset $M$ of $[0, \xi)$ such that $\mathcal{E}\cap \Theta_{\zeta+1}$ is cofinal in $\Theta_{\zeta+1}$ for each $\zeta\in M$. For $\delta>0$ and $\zeta\in M$,  we may apply the inductive hypothesis to $h|_{\Pi(\Theta_{\zeta+1})}$ to deduce the existence of $a_\zeta$ and $d_\zeta:\Theta_{\zeta+1}\to \Theta_{\zeta+1}$ with extension $e_\zeta$ such that $e_\zeta(MAX(\Theta_{\zeta+1}))\subset \mathcal{E}$, for each $(s,t)\in \Pi(\Theta_{\zeta+1})$, $h(d_\zeta(s), e_\zeta(t))\geqslant a_\zeta-\delta/2$, for each $t\in MAX(\Theta_{\zeta+1})$, $\sum_{\varnothing<s\leqslant e_\zeta(t)} h(s, e_\zeta(t))\leqslant a_\zeta+\delta/2$, and $(x_{d(t)})_{t\in \Theta_{\zeta+1}}$ is weakly null.  Since $h$ is bounded, $(a_\zeta)_{\zeta<\xi}$ is bounded, and we may fix $a\in \rr$ such that $N=\{\zeta\in M: |a-a_\zeta|<\delta/2\}$ is cofinal in $[0, \xi)$.  We now fix $d'_\eta$, $e'_\eta$ and then $d,e$ as in $(i)$ of Case $2$.

Case $3$: If for an ordinal $\xi$, $(\xi, n)$ holds for every $n\in \nn$, then $(\xi+1, 1)$ holds.

$(i)$ Assume $\mathcal{E}\subset MAX(\Gamma_{\xi+1, 1}.D)$ is cofinal. Then there exists a cofinal subset $M$ of $\nn$ such that for each $n\in M$, $\mathcal{E}\cap \Gamma_{\xi, n}.D$ is cofinal in $\Gamma_{\xi, n}.D$.   For each $n\in M$, there exists a monotone map $d_n:\Gamma_{\xi,n}.D\to \Gamma_{\xi,n}.D$ with extension $e_n$ such that $e_n(MAX(\Gamma_{\xi, n}.D))\subset \mathcal{E}$ and such that $(x_{d_n(t)})_{t\in \Gamma_{\xi,n}.D}$ is weakly null.  Now for each $i\in \nn$, fix $n_i\in M$ with $i\leqslant n_i$ and let $d'_i:\Gamma_{\xi,i}.D\to \Gamma_{\xi,n_i}.D$ be a monotone map such that $(x_{d_{n_i}\circ d_i'(t)})_{t\in \Gamma_{\xi, i}.D}$ is weakly null. Such a map $d_i'$ exists by Remark \ref{stupid remark}.   Let $e'_i$ be any extension of $d_i'$.  Define $d, e$ by $d|_{\Gamma_{\xi, i}.D}=d_{n_i}\circ d_i'$ and $e|_{MAX(\Gamma_{\xi,i}.D}= e_{n_i}\circ e_i'$.

$(ii)$ Assume $\mathcal{E}\subset MAX(\Gamma_{\xi+1, 1}.D)$ is cofinal.  Then there exists a cofinal subset $M$ of $\nn$ such that for each $n\in M$, $\mathcal{E}\cap \Gamma_{\xi,n}.D$ is cofinal in $\Gamma_{\xi,n}.D$.   By the inductive hypothesis, for each $n\in M$, there exists $1\leqslant j_n\leqslant k$ such that $\mathcal{E}_{j_n}\cap \Gamma_{\xi, n}.D$ is cofinal in $\Gamma_{\xi, n}.D$.     Let $M_j=\{n\in M: j=j_n\}$ and fix $1\leqslant j\leqslant k$ such that $M_j$ is cofinal in $\nn$. Then $\mathcal{E}_j\in \Omega_{\xi+1, 1}$.

$(iii)$ For each $n\in \nn$, applying the inductive hypothesis to $\chi|_{\Pi(\Gamma_{\xi, n}.D)}$ yields a level map $d_n:\Gamma_{\xi,n}.D\to \Gamma_{\xi, n}.D$ with extension $e_n$ and $\alpha^1_1, \ldots, \alpha^n_n\in F$ such that for each $1\leqslant i\leqslant n$ and $\Lambda_{\xi, n, i}.D\ni s \leqslant t\in MAX(\Gamma_{\xi, n}.D)$, $\chi(d_n(s), e_n(t))=\alpha^n_i$ and $(x_{d_n(t)})_{t\in \Gamma_{\xi, n}.D}$ is weakly null.  Then there exist $\alpha\in F$ and $1\leqslant n_1<n_2<\ldots$, and for each $i\in \nn$ $1\leqslant s^i_1<\ldots <s^i_i\leqslant n_i$ such that $\alpha^{n_i}_{s_j}=\alpha$ for each $1\leqslant j\leqslant i$. For each $i\in \nn$, fix a monotone map $d'_i:\Gamma_{\xi, i}.D\to \Gamma_{\xi, n_i}.D$ such that $d'_i(\Lambda_{\xi, i, j}.D)\subset \Lambda_{\xi, n_i, s^{n_i}_j}.D$ and such that $(x_{d_{n_i}\circ d_i'(t)})_{t\in \Gamma_{\xi, i}.D}$ is weakly null. Such a map exists by Proposition \ref{stupid1}.  Let $e'_i$ be any extension of $d_i'$. Let $d|_{\Gamma_{\xi, i}.D}=d_{n_i}\circ d_i'$ and $e|_{MAX(\Gamma_{\xi, i}.D)}= e_{n_i}\circ e_i'$.

$(iv)$  Fix a cofinal subset $M$ of $\nn$ such that for each $n\in M$, $\mathcal{E}\cap \Gamma_{\xi, n}.D$ is cofinal in $\Gamma_{\xi, n}.D$. For each $n\in M$, applying the inductive hypothesis to $h|_{\Pi(\Gamma_{\xi, n}.D)}$ yields a level map $d_n:\Gamma_{\xi, n}.D\to \Gamma_{\xi, n}.D$ with extension $e_n$ and numbers $a_1^n, \ldots, a^n_n$ such that $e_n(MAX(\Gamma_{\xi, n}.D))\subset \mathcal{E}$,  $(x_{d_n(t)})_{t\in \Gamma_{\xi, n}.D}$ is weakly null, for each $1\leqslant i\leqslant n$ and each $\Lambda_{\xi, n, i}.D\ni s\leqslant t\in MAX(\Gamma_{\xi, n}.D)$, $h(d_n(s), e_n(t))\geqslant a_i^n-\delta/2$, and for each $t\in MAX(\Gamma_{\xi,n}.D)$, $\sum_{\varnothing<s\leqslant e_n(t)} \mathbb{P}_{\xi, n}(s) h(s, e_n(t)) \leqslant \delta/2+\sum_{i=1}^n a_i^n$. Note that the collection $(a^n_i: n\in\nn, 1\leqslant i\leqslant n)$ is bounded, so there exist $n_1<n_2<\ldots$, $n_i\in M$,  such that $a=\lim_i \frac{1}{n_i}\sum_{j=1}^{n_i} a^{n_i}_j$ exists.  By passing to a further subsequence of $(n_i)_{i=1}^\infty$, we may assume that for each $i\in \nn$, $$|\{j\leqslant n_i: a^{n_i}_j\geqslant a-\delta/2\}|\geqslant i$$ and $a+\delta/2 \geqslant \frac{1}{n_i}\sum_{j=1}^{n_i} a^{n_i}_j$.  For each $i$, we fix $1\leqslant s^{n_i}_1<\ldots <s^{n_i}_i\leqslant n_i$ such that $a^{n_i}_{s_j}\geqslant a-\delta/2$ for each $1\leqslant j\leqslant i$.  We now fix $d'_i$, $e'_i$ and finish as in $(i)$ of Case $3$.

Case $4$: If $(\xi, k)$ holds for some $\xi$ and each $1\leqslant k\leqslant n$, $(\xi, n+1)$ holds. 

$(i)$ Suppose $\mathcal{E}\subset MAX(\Gamma_{\xi, n+1}.D)$ is cofinal. For each $t\in MAX(\Lambda_{\xi, n+1, 1}.D)$, let $P_t=\{s\in \Gamma_{\xi, n+1}.D: t<s\}$. Then there exists a subset $\mathcal{F}$ of $MAX(\Lambda_{\xi, n+1, 1}.D)$ which is cofinal in $\Lambda_{\xi, n+1, 1}.D$ such that $\mathcal{E}\cap P_t$ is cofinal in $P_t$ for each $t\in \mathcal{F}$. For each $t\in \mathcal{F}$, fix a level map $d_t:P_t\to P_t$ with extension $e_t$ such that $e_t(MAX(P_t))\subset \mathcal{E}$ and $(x_{d_t(s)})_{s\in P_t}$ is weakly null. Now fix a monotone map $d':\Lambda_{\xi, n+1, 1}.D\to \Lambda_{\xi, n+1, 1}.D$ with extension $e':MAX(\Lambda_{\xi, n+1, 1}.D)\to MAX(\Lambda_{\xi, n+1, 1}.D)$ such that $e'(MAX(\Lambda_{\xi, n+1, 1}.D))\subset \mathcal{F}$ and $(x_{d'(t)})_{t\in \Lambda_{\xi, n+1, 1}.D}$ is weakly null. For each $t\in MAX(\Lambda_{\xi, n+1, 1}.D)$, let $j_t:P_t\to P_{e'(t)}$ be the canonical identification.  Now define $d$, $e$ by letting $d|_{\Lambda_{\xi, n+1, 1}.D}=d'$ and for $t\in MAX(\Lambda_{\xi, n+1, 1}.D)$, $d|_{P_t}= d_{e'(t)}\circ j_t$ and $e|_{MAX(P_t)}= e_{e'(t)}\circ j_t$.

$(ii)$ For each $t\in MAX(\Lambda_{\xi, n+1, 1}.D)$, let $P_t$ be as in $(i)$. Then there exists a set $\mathcal{F}\subset MAX(\Lambda_{\xi, n+1, 1}.D)$ which is cofinal in $\Lambda_{\xi, n+1, 1}.D$ such that $\mathcal{E}\cap P_t$ is cofinal in $P_t$ for each $t\in \mathcal{F}$. Applying the inductive hypothesis, for each $t\in \mathcal{F}$, there exists $1\leqslant j_t\leqslant k$ such that $\mathcal{E}_{j_t}\cap P_t$ is cofinal in $P_t$. For each $1\leqslant j\leqslant k$, let $\mathcal{F}_j=\{t\in \mathcal{F}: j=j_t\}$. Then by the inductive hypothesis, there exists $1\leqslant  j\leqslant k$ such that $\mathcal{F}_j$ is cofinal in $\Lambda_{\xi, n+1, 1}.D$, whence $\mathcal{E}_j\in \Omega_{\xi, n+1}$.

$(iii)$ For each $t\in MAX(\Lambda_{\xi, n+1, 1}.D)$, let $P_t$ be as in $(i)$. Applying the inductive hypothesis to $\chi|_{\Pi(P_t)}$ yields a level pruning $d_t:P_t\to P_t$ with extension $e_t$ and $(\alpha^t_i)_{i=2}^{n+1}\in F^n$ such that for each $2\leqslant i\leqslant n+1$ and each $P_t\cap \Lambda_{\xi, n+1, i}.D\ni s\leqslant u\in MAX(P_t)$, $\chi(d_t(s), e_t(u))=\alpha^t_i$ and such that $(x_{d(s)})_{s\in P_t}$ is weakly null.   Now, with $t\in MAX(\Lambda_{\xi, n+1, 1}.D)$ still fixed, for each $\beta\in F^{|t|}$, let $\mathcal{E}_\beta= \{s\in MAX(P_t): (\chi(t|_i, e_t(s)))_{i=1}^{|t|}= \beta\}$. By $(i)$ and $(ii)$, we may fix a level map $d'_t:P_t\to P_t$ with extension $e'_t$ and $\beta_t=(\beta^t_i)_{i=1}^{|t|}$ such that $e'_t(MAX(P_t))\subset \mathcal{E}_{\beta_t}$ and $(x_{d_t\circ d'_t(s)})_{s\in P_t}$ is weakly null.  We now note that for any $s\leqslant t\in MAX(\Lambda_{\xi, n+1, 1}.D)$ and any maximal extension $u$ of $t$, $\chi(s, e_t\circ e_t'(u))=\beta^t_{|s|}$.  Now define $\upsilon:\Pi(\Lambda_{\xi, n+1, 1}.D)\to F$ by letting $\upsilon(s, t)= \beta^t_{|s|}$.  By the inductive hypothesis, there exist a monotone map $d'':\Lambda_{\xi, n+1, 1}.D\to \Lambda_{\xi, n+1, 1}.D$ with extension $e''$ and $\alpha_1\in F$ such that for each $(s,t)\in \Pi(\Lambda_{\xi, n+1, 1}.D)$, $\upsilon(d''(s), e''(t))=\alpha_1$ and $(x_{d''(t)})_{t\in \Lambda_{\xi, n+1, 1}.D}$ is weakly null. Now for each $\beta\in F^n$, let $\mathcal{F}_\beta=\{t\in MAX(\Lambda_{\xi, n+1, 1}.D): (\alpha^{e''(t)}_i)_{i=2}^{n+1}=\beta\}$. By $(i)$ and $(ii)$, there exist another monotone map $d''':\Lambda_{\xi, n+1, 1}.D\to \Lambda_{\xi, n+1, 1}.D$ with extension $e'''$ and $\beta=(\alpha_i)_{i=2}^{n+1}$ such that $e'''(\Lambda_{\xi, n+1, 1}.D)\subset \mathcal{F}_\beta$.  Now define $d$ and $e$ by letting $d|_{\Lambda_{\xi, n+1, 1}.D}= d''\circ d'''$, $d|_{P_t}= d_{e''\circ e'''(t)}\circ d'_{e''\circ e'''(t)}\circ j_t$ and $e|_{MAX(P_T)}= e_{e''\circ e'''(t)}\circ e'_{e''\circ e'''(t)}\circ j_t$, where $j_t:P_t\to P_{e''\circ e'''(t)}$ is the canonical identification.

$(iv)$ This is quite similar to the previous paragraph. For each $t\in MAX(\Lambda_{\xi, n+1, 1}.D)$, let $P_t$ be as in $(i)$. Note that there exists a subset $\mathcal{F}$ of $MAX(\Lambda_{\xi, n+1, 1}.D)$ which is cofinal in $\Lambda_{\xi, n+1, 1}.D$ such that $\mathcal{E}\cap P_t$ is cofinal in $P_t$ for each $t\in \mathcal{F}$. For each $t\in \mathcal{F}$, applying the inductive hypothesis to $h|_{\Pi(P_t)}$ yields a level pruning $d_t:P_t\to P_t$ with extension $e_t$ and $(a^t_i)_{i=2}^{n+1}\in \rr^n$ such that for each $2\leqslant i\leqslant n+1$ and each $P_t\cap \Lambda_{\xi, n+1, i}.D\ni s\leqslant u\in MAX(P_t)$, $h(d_t(s), e_t(u))\geqslant a^t_i-\delta/4$, for any $s\in MAX(P_t)$, $\sum_{t<u\leqslant e_t(s)} h(u, e_t(s)) \leqslant \delta/4+\sum_{i=2}^{n+1} a^t_i$, and $(x_{d_t(s)})_{s\in P_t}$ is weakly null.  Now, with $t\in \mathcal{F}$ still fixed, fix a finite partition  $\mathcal{U}_t$ of subsets $\beta$ of $\text{range}(h)^{|t|}$ such that each $\beta\in \mathcal{U}_t$ has diameter (with respect to the $\ell_1^{|t|}$ on $\rr^{|t|}$) less than $\delta/4$. For each $\beta\in \mathcal{U}_t$,  let $\mathcal{E}_\beta= \{s\in MAX(P_t): (h(t|_i, e_t(s)))_{i=1}^{|t|}\in \beta\}$. By $(i)$ and $(ii)$, we may fix a level map $d'_t:P_t\to P_t$ with extension $e'_t$ and $\beta_t=(\beta^t_i)_{i=1}^{|t|}$ such that $e'_t(MAX(P_t))\subset \mathcal{E}_{\beta_t}$ and $(e_{d_t\circ d'_t(s)})_{s\in P_t}$ is weakly null.  Now fix $(\varpi^t_i)_{i=1}^{|t|}\in \beta_t$. We now note that for any $s\leqslant t\in MAX(\Lambda_{\xi, n+1, 1}.D)$ and any maximal extension $u$ of $t$, $|\varpi^t_{|s|}-h(s, e_t\circ e_t'(u))|<\delta/4$.   Now define $\upsilon:\Pi(\Lambda_{\xi, n+1, 1}.D)\to \rr$ by letting $\upsilon(s, t)= \varpi^t_{|s|}$ if $t\in \mathcal{F}$, and $\upsilon(s,t)=0$ otherwise.  By the inductive hypothesis, there exist $a_1\in \rr$ and a monotone map $d'':\Lambda_{\xi, n+1, 1}.D\to \Lambda_{\xi, n+1, 1}.D$ with extension $e'':\Lambda_{\xi, n+1, 1}.D\to \mathcal{F}$ such that for each $(s,t)\in \Pi(\Lambda_{\xi, n+1, 1}.D)$, $\upsilon(d''(s), e''(t))\geqslant a_1-\delta/4$, for each $t\in MAX(\Lambda_{\xi, n+1, 1}.D)$, $\sum_{\varnothing<s\leqslant e''(t)} \upsilon(s, e''(t)) \leqslant a_1+\delta/4$, and  $(x_{d''(t)})_{t\in \Lambda_{\xi, n+1, 1}.D}$ is weakly null. Now by boundedness of the collection $(a^t_i: 2\leqslant i\leqslant n+1, t\in MAX(\Lambda_{\xi, n+1, 1}.D))$, we may fix a finite subset $S$ of $\rr^n$ such that for each $t\in MAX(\Lambda_{\xi, n+1, 1}.D)$, there exists $\gamma_t=(b^t_i)_{i=2}^{n+1}\in S$ such that $\sum_{i=2}^{n+1}| a^{e''(t)}_i- b^t_i|<\delta/4$. Now fix another monotone map $d''':\Lambda_{\xi, n+1, 1}.D\to \Lambda_{\xi, n+1, 1}.D$ and $(a_i)_{i=2}^{n+1}\in S$ such that $(a_i)_{i=2}^{n+1}= (b^{e'''(t)}_i)_{i=2}^{n+1}$ for all $t\in MAX(\Lambda_{\xi, n+1, 1}.D)$.  Now define $d$ and $e$ by letting $d|_{\Lambda_{\xi, n+1, 1}.D}= d''\circ d'''$, $d|_{P_t}= d_{e''\circ e'''(t)}\circ d'_{e''\circ e'''(t)}\circ j_t$ and $e|_{MAX(P_T)}= e_{e''\circ e'''(t)}\circ e'_{e''\circ e'''(t)}\circ j_t$, where $j_t:P_t\to P_{e''\circ e'''(t)}$ is the canonical identification.

\end{proof}

Our final task in this section is to prove Proposition \ref{lame claim}, contained in Lemma \ref{ks}.  We first make the following observation. 

\begin{rem}\upshape Let $A_0:X_0\to Y_0$, $A_1:X_1\to Y_1$ be non-zero operators and let $R=\max\{ \|A_0\|, \|A_1\|\}$.   Suppose also that $y^*_0, v_0^* \in B_{Y^*_0}$, $y^*_1, v^*_1\in B_{Y^*_1}$, $\gamma'\in \rr$ are such that $$\|A^*_0 y^*_0\otimes A^*_1 y^*_1- A^*_0v^*_0\otimes A^*_1 v^*_1\|> \gamma'.$$ Then \begin{align*} \gamma' & < \|A^*_0 y^*_0\otimes A^*_1 y^*_1- A^*_0v^*_0\otimes A^*_1 v^*_1\| \\ & \leqslant \|A^*_0 y^*_0\otimes A^*_1 y^*_1- A^*_0v^*_0\otimes A^*_1 y^*_1\|  + \|A^*_0 v^*_0\otimes A^*_1 y^*_1- A^*_0v^*_0\otimes A^*_1 v^*_1\| \\ & =  \|A^*_0 y^*_0- A^*_0v^*_0\|\|A^*_1 y^*_1\|  + \|A^*_0v^*_0\|\| A^*_1 y^*_1-  A^*_1 v^*_1 \| \\ & \leqslant 2R\max\{\|A^*_0 y^*_0-A^*v_0^*\|, \|A^*_1 y^*_1-A^*_1v^*_1\|\}.\end{align*}

\label{stupid fact} 
\end{rem}

We will need the following.

\begin{proposition} Suppose $A:X\to Y$ is an operator, $\sigma, \tau>0$, $\xi$ is an ordinal, and $\delta_\xi^{\text{\emph{weak}}^*}(\tau, A)\geqslant \sigma \tau$. If $T$ is a rooted tree with $o(T)=\omega^\xi+1$ and if $c\geqslant \tau$ and if $(y^*_t)_{t\in T}\subset B_{Y^*}$ is a weak$^*$-closed tree such that $\|A^*y^*_t- A^*y^*_{t^-}\|\geqslant c$ for all $t\in T.D$, then $\sigma c\leqslant 1$ and $\|y^*_\varnothing\|\leqslant 1-c\sigma$. 

\label{tc}

\end{proposition}

\begin{proof} It was shown in \cite[Proposition $3.10$]{CD} that under these hypotheses, either $y^*=0$ or $\|y^*\|\leqslant 1-\sigma c$.  Thus it suffices to prove the result in the case that $y^*=0$.  As was explained in \cite{CD}, the conditions imply that $Sz(A)>\omega^\xi$, whence $\delta^{\text{weak}^*}(\cdot, A)$ is finite and continuous. Fix $0<\sigma_1<\sigma$ and $0<\tau_1<\tau$ such that $\delta^{\text{weak}^*}(\tau_1, A)\geqslant \sigma_1\tau_1$. Fix $\delta>0$ such that $(1+\delta)\tau_1<\tau$.  Fix any $z^*\in Y^*$ with $\|z^*\|=\delta$. Let $z^*_t=(1+\delta)^{-1}(z^*+y^*_t)$. Then we deduce from the result in \cite{CD} that $\sigma_1 \cdot \frac{c}{1+\delta}\leqslant 1$. Since $0<\sigma_1$ and $\delta>0$ were arbitrary, $\sigma c\leqslant 1$.

\end{proof}

We now introduce another notion which is closely related to the Szlenk index, but defined to overcome the deficiency that the adjoint of an operator need not be injective. Given an operator $B:E\to F$, a weak$^*$-compact subset $K$ of $F^*$, and $\ee>0$, we define $\langle K\rangle_{B, \ee}$ denote the set of those $f^*\in K$ such that for every weak$^*$-neighborhood $V$ of $f^*$, $\text{diam}(B^*(K\cap V))>\ee$. We define the transfinite derivations $$\langle K\rangle_{B, \ee}^0=K,$$ $$\langle K\rangle_{B, \ee}^{\xi+1}= \langle \langle K\rangle_{B, \ee}^\xi\rangle_{B, \ee},$$ and if $\xi$ is a limit ordinal, $$\langle K\rangle_{B, \ee}^\xi=\bigcap_{\zeta<\xi} \langle K\rangle_{B, \ee}.$$  Note that $\langle K\rangle_{B, \ee}^\xi$ is weak$^*$-compact for any ordinal $\xi$ and any $\ee>0$.  We note that if $B:E\to F$ is the identity of $E$, then the notions above exactly coincide with the Szlenk derivations $s_\ee^\xi$. 

The following is a trivial proof by induction and is well-known for the Szlenk index. Our proof is an inessential modification. 

\begin{lemma} If $B:E\to F$ is an operator, $K\subset F^*$ is weak$^*$-compact, $\ee>0$, and $f^*\in \langle K\rangle_{B, \ee}^\xi$, then there exists a rooted tree $T$ with $o(T)=\xi+1$ and a weak$^*$-closed collection $(f^*_t)_{t\in T}\subset K$ such that $f^*_\varnothing=f^*$ and for each $\varnothing\neq t\in T$, $\|B^*f^*_t- B^*f^*_{t^-}\|>\ee/2$.

\label{common sense}

\end{lemma}

\begin{lemma} Suppose $T$ is a non-empty, well-founded tree, $B:E\to F$ is an operator, $D$ is a directed set, $\gamma>0$,  $(e_t)_{t\in T.D}\subset B_E$ is weakly null, $K\subset F^*$ is weak$^*$-compact, and $(f^*_t)_{t\in MAX(T.D)}\subset K$ is such that for any $\varnothing<s\leqslant t\in MAX(T)$, $$\text{\emph{Re\ }} f^*_t(Be_s) \geqslant \gamma.$$   Then for any $0<\gamma'<\gamma$, $\langle K\rangle_{B, \gamma'}^{o(T)}\neq \varnothing$. 

\label{payne}

\end{lemma}

\begin{proof} We prove by induction on $\xi$ that if $t\in T^\xi.D$, there exists $f^*\in \langle K\rangle_{B, \gamma'}^\xi$ such that for any $\varnothing<s\leqslant t$, $\text{Re\ }f^*(Be_s)\geqslant \gamma$. The $\xi=0$ case holds by the hypotheses.  Suppose $\xi$ is a limit ordinal and $t\in T^\xi.D$. Then for every $\zeta<\xi$, we may fix  $f^*_\zeta\in \langle K\rangle_{B, \gamma'}^\zeta$ such that $\text{Re\ }f^*_\zeta(Be_s)\geqslant \gamma$ for each $ \varnothing<s\leqslant t$.  Then any weak$^*$-limit $f^*$  of a weak$^*$-converging subnet of $(f^*_\zeta)_{\zeta<\xi}$ lies in $\langle K\rangle^\xi_{B, \gamma'}$ and satisfies $\text{Re\ }f^*(Be_s)\geqslant \gamma$ for each $\varnothing<s\leqslant t$.  Last suppose we have the result for some ordinal $\xi$ and $t\in T^{\xi+1}.D$. Then there exists some $\lambda$ such that $t\smallfrown (\lambda, u)\in T^\xi.D$ for some (equivalently, every) $u\in D$. Then for each $u\in D$, we fix $f^*_u\in \langle K\rangle_{B, \gamma'}^\xi$ such that $\text{Re\ }f^*_u(e_{t\smallfrown (\lambda, u)})\geqslant \gamma$ and $\text{Re\ }f^*_u (e_s)\geqslant \gamma$ for each $\varnothing<s\leqslant t$.  Let $f^*$ be a weak$^*$-limit of a subnet $(f^*_u)_{u\in D'}$ of $(f^*_u)_{u\in D}$.   Then since $$\underset{u\in D'}{\lim\inf} \|B^*f^*_u -B^* f^*\|\geqslant \underset{u\in D'}{\lim\inf} \text{Re\ }(B^*f^*_u-B^*f^*_u)(e_{t\smallfrown (\lambda, u)})\geqslant \gamma>\gamma',$$ we deduce that $f^*\in \langle K\rangle_{B, \gamma'}^{\xi+1}$.

Now if $o(T)$  is a limit ordinal, for any $\xi<o(T)$, $\langle K\rangle_{B, \gamma'}^\xi\neq \varnothing$ by the previous paragraph, whence $\langle K\rangle_{B, \gamma'}^{o(T)}\neq \varnothing$ by weak$^*$-compactness.    If $o(T)=\xi+1$, then we may fix $\lambda$ such that $(\lambda, u)\in T^\xi.D$ for some (equivalently, all) $u\in D$ and for each $u\in D$, fix $f^*_u$ such that $\text{Re\ }f^*_u(e_{(\lambda, u)})\geqslant \gamma$. Arguing as in the successor case of the previous paragraph, we deduce that if $f^*$ is a weak$^*$-limit of a subnet of $(f^*_u)_{u\in D}$, $f^*\in \langle K\rangle_{B, \gamma'}^{\xi+1}= \langle K\rangle_{B, \gamma'}^{o(T)}$.

\end{proof}

For Banach spaces $Y_0$, $Y_1$, $K_0\subset Y^*_0$, $K_1\subset Y^*_1$, we let $$[K_0, K_1]=\{y^*_0\otimes y^*_1: y^*_0\in K_0, y^*_1\in K_1\}\subset (Y_0\hat{\otimes}_\ee Y_1)^*.$$

\begin{lemma} Suppose $A_0:X_0\to Y_0$ $A_1:X_1\to Y_1$ are non-zero operators, $\ee>0$, $\xi$ is a limit ordinal, $K_0\subset B_{Y^*_0}$, $K_1\subset B_{Y^*_1}$ are weak$^*$-compact, and $M\subset [0, \xi)$ is cofinal in $[0, \xi)$.  Then $$\bigcap_{\zeta\in M} [\langle K_0\rangle_{A_0, \ee}^\zeta, K_1]\subset [\langle K_0\rangle_{A_0, \ee}^\xi, K_1]$$ and $$\bigcap_{\zeta\in M } [K_0, \langle K_1\rangle_{A_1, \ee}^\zeta]\subset [K_0, \langle K_1\rangle_{A_1, \ee}^\xi].$$ 

\label{liotta}

\end{lemma}

\begin{proof} Define $j:B_{(Y_0\oplus_1 Y_1)^*}\to B_{(Y_0\hat{\otimes}_\ee Y_1)^*}$ by $j(y^*_0, y^*_1)=y^*_0\otimes y^*_1$. Note that $j$ is weak$^*$-weak$^*$ continuous. We prove the first containment, with the second containment following by symmetry.  Assume $u^*\in \cap_{\zeta\in M} [\langle K_0\rangle_{A_0, \ee}^\zeta, K_1]$ and for each $\zeta\in M$, fix $y^{*, \zeta}_0\in \langle K_0\rangle_{A_0, \ee}^\zeta$ and $y^{*, \zeta}_1\in K_1$ such that $u^*=y^{*, \zeta}_0\otimes y^{*, \zeta}_1= j(y^{*, \zeta}_0, y^{*, \zeta}_1)$.   Now fix $$(y^*_0,y^*_1)\in \bigcap_{\zeta<\xi} \overline{\{(y^{*, \gamma}_0, y^{*, \gamma}_1): \zeta\leqslant \gamma\in M\}}^{\text{weak}^*} \subset \bigcap_{\zeta<\xi} (\langle K_0\rangle_{A_0, \ee}^\zeta\times K_1)= \langle K_0 \rangle_{A_0, \ee}^\xi\times K_1.$$  Since $j$ is weak$^*$-weak$^*$ continuous, $j|_{\overline{\{(y^{*, \gamma}_0, y^{*, \gamma}_1): \zeta\leqslant \gamma\in M\}}^{\text{weak}^*} }\equiv u^*$ and  $u^*=y^*_0\otimes y^*_1\in [\langle K_0\rangle_{A_0, \ee}^\xi, K_1]$.

\end{proof}

\begin{lemma} Let $A_0:X_0\to Y_0$, $A_1:X_1\to Y_1$ be non-zero operators, $R=\max\{\|A_0\|, \|A_1\|\}$. Let $J$ be a finite set and suppose that for each $j\in J$, $K_{0,j}\subset B_{Y^*_0}$ is a weak$^*$-compact set and $K_{1, j}\subset B_{Y^*_1}$ is a weak$^*$-compact set. 

Then for any $\ee>0$, any ordinal $\xi$, and any $n\in \nn$, $$\Biggl\langle \bigcup_{j\in J} [K_{0,j}, K_{1,j}]\Biggr\rangle_{A_0\otimes A_1, \ee}^{\omega^\xi n} \subset \bigcup_{j\in J}\bigcup_{m=0}^n [\langle K_{0,j}\rangle_{A_0, \ee/4R}^{\omega^\xi m}, \langle K_{1,j}\rangle_{A_1, \ee/4R}^{\omega^\xi(n-m)}].$$

\end{lemma}

\begin{proof} As usual, we will work by induction on $\ord\times \nn$ with its lexicographical ordering. Assume $u^*\in \Bigl\langle \bigcup_{j\in J} [K_{0,j}, K_{1,j}]\Bigr\rangle_{A_0\otimes A_1, \ee}$.  This means there exists a net $(y^*_{0, \lambda}\otimes y^*_{1, \lambda})_\lambda \in \bigcup_{j\in J} [K_{0,j}, K_{1,j}]$ converging weak$^*$ to $u^*$ and such that $\|(A_0\otimes A_1)^*u^*-(A_0\otimes A_1)^*y^*_{0, \lambda}\otimes y^*_{1, \lambda}\|>\ee/2$ for all $\lambda$. By passing to subnets, we may suppose there exists $j\in J$ such that $y^*_{0, \lambda}\in K_{0,j}$ and $y^*_{1, \lambda}\in K_{1,j}$ for all $\lambda$. By passing to a further subnet, we may assume $y^*_{0, \lambda}\underset{\text{weak}^*}{\to} y^*_0\in K_{0,j}$ and $y^*_{1, \lambda}\underset{\text{weak}^*}{\to} y^*_1\in K_{1,j}$.   Then $u^*=y^*_0\otimes y^*_1$.  By Remark \ref{stupid fact} and by passing to a subnet once more and swithcing $A_0$ and $A_1$ if necessary, we may assume $\|A^*_0 y^*_{0, \lambda}- A^*_0y^*_0\|>\ee/4R$ for all $\lambda$. This shows that $u^*\in [\langle K_{0,j}\rangle_{A_0, \ee/4R}^1, \langle K_{1,j}\rangle_{A_1, \ee/4R}^0]$. This finishes the $(\xi,n)=(0,1)$ case.

Now assume $\xi$ is a limit ordinal and the claim holds for every $\zeta<\xi$. Fix $u^*\in \Bigl\langle \bigcup_{j\in J} [K_{0,j}, K_{1,j}]\Bigr\rangle_{A_0\otimes A_1, \ee}^{\omega^\xi}$. For every $\zeta<\xi$, there exist $j_\zeta\in J$ and $n_\zeta\in \{0,1\}$ such that $u^*\in [\langle K_{0, j_\zeta}\rangle_{A_0, \ee/4R}^{\omega^\zeta n_\zeta}, \langle K_{1, j_\zeta}\rangle_{A_1, \ee/4R}^{\omega^\zeta (1-n_j)}]$. There exist a cofinal subset $M$ of $[0, \xi)$, $j\in J$, and $n\in \{0, 1\}$ such that $j_\zeta=j$ and $n_\zeta=n$ for all $\zeta\in M$. It then follows from Lemma \ref{liotta} that $u^*\in [\langle K_{0, j}\rangle_{A_0, \ee/4R}^{\omega^\xi n}, \langle K_{1,j}\rangle_{A_1, \ee/4R}^{\omega^\xi(1-n)}]$, which yields the $(\xi, 1)$, $\xi$ a limit case.

Now assume that for some ordinal $\xi$ and every $n\in \nn$, the claim holds for $(\xi, n)$. Then if $u^*\in \Bigl\langle \bigcup_{j\in J} [K_{0,j}, K_{1,j}]\Bigr\rangle_{A_0\otimes A_1, \ee}^{\omega^{\xi+1}}$, for each $n\in \nn$,  \begin{align*} u^* & \in \Bigl\langle \bigcup_{j\in J} [K_{0,j}, K_{1,j}]\Bigr\rangle_{A_0\otimes A_1, \ee}^{\omega^\xi(2n-1)} \subset \bigcup_{j\in J}\bigcup_{m=0}^n [\langle K_{0,j}\rangle_{A_0, \ee/4R}^{\omega^\xi m}, \langle K_{1,j}\rangle_{A_1, \ee/4R}^{\omega^\xi(2n-1-m)}] \\ & \subset \bigcup_{j\in J} [\langle K_{0,j}\rangle_{A_0, \ee/4R}^{\omega^\xi n}, K_{1,j}]\cup [K_{0,j}, \langle K_{1,j}\rangle_{A_1, \ee/4R}^{\omega^\xi n}]. \end{align*} Thus for every $n\in \nn$, we may fix $j_n\in J$ and $k_n\in \{0, 1\}$ such that $$u^*\in [\langle K_{0, j_n}\rangle_{A_0, \ee/4R}^{\omega^\xi n k_n}, \langle K_{1,j_n}\rangle_{A_1, \ee/4R}^{\omega^\xi n(1-k_n)}].$$ Then there exist a cofinal subset $M$ of $\nn$, $j\in J$, and $k\in \{0,1\}$ such that $j_n=j$ and $k_n=k$ for all $n\in M$. By Lemma \ref{liotta},  $$u^*\in [\langle K_{0,j}\rangle_{A_0, \ee/4R}^{\omega^{\xi+1}}, K_{1,j}]$$ if $k=1$ and $$u^*\in [K_{0,j}, \langle K_{1,j}\rangle_{A_1, \ee/4R}^{\omega^{\xi+1}}]$$ if $k=0$.

Now assume that for some ordinal $\xi$ some $n\in \nn$, and every $1\leqslant k\leqslant n$, we have the result for the pair $(\xi, k)$. Then \begin{align*} \Bigl\langle \bigcup_{j\in J}[K_{0,j}, K_{1,j}]\Bigr \rangle_{A_0\otimes A_1, \ee}^{\omega^\xi(n+1)} & = \Bigl \langle \Bigl\langle \bigcup_{j\in J}[K_{0,j}, K_{1,j}] \Bigr\rangle_{A_0\otimes A_1, \ee}^{\omega^\xi n}\Bigr\rangle_{A_0\otimes A_1, \ee}^{\omega^\xi} \\ & \subset \Bigl\langle \bigcup_{j\in J}\bigcup_{m=0}^n [\langle K_{0,j}\rangle_{A_0, \ee/4R}^{\omega^\xi m}, \langle K_{1,j}\rangle_{A_1, \ee/4R}^{\omega^\xi(n-m)}]\Bigr \rangle_{A_0\otimes A_1, \ee}^{\omega^\xi} \\ & \subset \bigcup_{j\in J} \bigcup_{m=0}^n [\langle \langle K_{0,j}\rangle_{A_0, \ee/4R}^{\omega^\xi m}\rangle_{A_0, \ee/4R}^{\omega^\xi},\langle K_{1,j}\rangle_{A_1, \ee/4R}^{\omega^\xi(n-m)} ] \\ &  \cup [\langle K_{0,j}\rangle_{A_0, \ee/4R}^{\omega^\xi m}, \langle \langle K_{1,j}\rangle_{A_1, \ee/4R}^{\omega^\xi(n-m)}\rangle_{A_1, \ee/4R}^{\omega^\xi}] \\ & =\bigcup_{j\in J}\bigcup_{m=0}^{n+1} [\langle K_{0,j}\rangle_{A_0,\ee/4R}^{\omega^\xi m}, \langle K_{1,j}\rangle_{A_1, \ee/4R}^{\omega^\xi(n+1-m)}].\end{align*}

\end{proof}

\begin{proof}[Proof of Proposition  \ref{lame claim}]
We recall the progress we had made prior to the statement of Proposition \ref{lame claim}.  We have  $u\in B_{Y_0\hat{\otimes}_\ee Y_1}$, $(u_{d\circ d'(s)})_{s\in \Gamma_{\xi,1}.D}\subset \frac{\sigma}{8R} B_{X_0\hat{\otimes}_\ee X_1}$, $(y^*_{0, e\circ e'(t)})_{t\in MAX(\Gamma_{\xi,1}.D)}\subset B_{Y^*_0}$, $(y^*_{1, e\circ e'( t)})_{t\in MAX(\Gamma_{\xi,1}.D)}\subset B_{Y^*_1}$,  $\alpha, \beta, \delta$ such that $\beta-2\delta>0$, $(u_{d\circ d'(s)})_{s\in \Gamma_{\xi,1}.D}$ is weakly null, $|\alpha- \text{Re\ }y^*_{0, e\circ e'(t)}\otimes y^*_{1, e\circ e'(t)}(u)|\leqslant \delta$ for all $t\in MAX(\Gamma_{\xi,1}.D)$, and for each $(s,t)\in \Pi(\Gamma_{\xi,1}.D)$, $|\beta-\text{Re\ }y^*_{0, e\circ e'(t)}\otimes y^*_{1, e\circ e'(t)} (A_0\otimes A_1(\frac{8R}{\sigma} u_{d\circ d'(s)}))|\leqslant \delta$.  Let $$K=\overline{\{y^*_{0, e\circ e'(t)}\otimes y^*_{1, e\circ e'(t)}: t\in MAX(\Gamma_{\xi,1}.D)\}}^{\text{weak}^*}\subset [B_{Y^*_0}, B_{Y^*_1}]$$ and note that, by Lemma \ref{payne} applied with $\gamma=\beta-\delta$ and $z_t=\frac{8R}{\sigma}u_{d\circ d'(t)}$,  $\langle K\rangle_{A_0\otimes A_1, \beta-2\delta}^{\omega^\xi}\neq \varnothing$. Fix $u^*\in \langle K\rangle_{A_0\otimes A_1, \beta-2\delta}^{\omega^\xi}$ and note that $$u^*\in \langle [B_{Y^*_0}, B_{Y^*_1}]\rangle_{A_0\otimes A_1, \beta-2\delta}^{\omega^\xi} \subset [\langle B_{Y^*_0}\rangle_{A_0, \frac{\beta-2\delta}{4R}}^{\omega^\xi}, B_{Y^*_1}] \cup [B_{Y^*_0}, \langle B_{Y^*_1}\rangle_{A_1, \frac{\beta-2\delta}{4R}}^{\omega^\xi}].$$

By Lemma \ref{common sense} and Proposition \ref{tc} applied with $c= \frac{\beta-2\delta}{8R}\geqslant \tau$,  $$[\langle B_{Y^*_0}\rangle_{A_0, \frac{\beta-2\delta}{4R}}^{\omega^\xi}, B_{Y^*_1}] \cup [B_{Y^*_0}, \langle B_{Y^*_1}\rangle_{A_1, \frac{\beta-2\delta}{4R}}^{\omega^\xi}]\subset (1-\sigma c)B_{(Y_0\hat{\otimes}_\ee Y_1)^*},$$ and $\|u^*\|\leqslant 1-\sigma(\frac{\beta-2\beta}{8R})$.    Since $u^*\in K$, $|\alpha - \text{Re\ }u^*(u)| \leqslant \delta$, whence $\alpha\leqslant \delta+ \text{Re\ }u^*(u)\leqslant 1+\delta +\sigma (\frac{\beta-2\delta}{8R})$.

\end{proof}


\begin{thebibliography}{HD}

\normalsize
\baselineskip=17pt

\bibitem{BKL} F. Baudier, N. Kalton, G. Lancien, \emph{A new metric invariant for Banach spaces}, Studia Math., \textbf{199} (2010), 73-94.

\bibitem{BP} C. Bessaga, A. Pe \l czy\'{n}ski, \emph{Spaces of continuous functions IV}, Studia Math. \textbf{19} (1960),
53-62.

\bibitem{Br} P. A. H. Brooker, \emph{Asplund operators and the Szlenk index}, J. Operator Theory \textbf{68} (2012), 405-442.

\bibitem{CAlt} R. M. Causey, \emph{An alternate description of the Szlenk index with applications}, Illinois J. Math. \textbf{59}  (2015),  no. 2, 359-390.

\bibitem{C} R. M. Causey, \emph{The Szlenk index of convex hulls and injective tensor products}, to appear in Journal of Functional Analysis. 



\bibitem{C3} R. M. Causey, \emph{Power type $\xi$-asymptotically uniformly smooth norms}, to appear in Transactions of the American Mathematical Society. 

\bibitem{C4} R. M. Causey, \emph{Power type asymptotically uniformly smooth and asymptotically uniformly flat norms}, submitted. 

\bibitem{CD} R. M. Causey, S. J. Dilworth, \emph{$\xi$-asymptotically uniformly smooth, $\xi$-asymptotically uniformly convex, and $(\beta)$ operators}, submitted. 

\bibitem{DKLR} S. J. Dilworth, D. Kutzarova, G. Lancien, L. Randrianarivony, \emph{Equivalent norms with property $(\beta)$ of Rolewicz}, Rev. R. Acad. Cienc. Exactas F\'{ı}s. Nat. Ser. A Math. RACSAM, \textbf{111} (2017), no. 1, 101-113.

\bibitem{GKL} G. Godefroy, N. Kalton, G. Lancien, \emph{Szlenk indices and uniform homeomorphisms}, Trans. Amer. Math. Soc., \textbf{353} (2001) no. 1., 3895-3918.

\bibitem{GKL2} G. Godefroy, N. Kalton, G. Lancien, \emph{Subspaces of $c_0(\nn)$ and Lipschitz isomorphisms}, Geom. Funct. Anal., \textbf{10} (2000), 798-820.

\bibitem{HL} P. H\'{a}ajek, G. Lancien, \emph{Various slicing indices on Banach spaces}, Mediterr. J. Math.
\textbf{4}(2007), 179-190.


\bibitem{KOS} H. Knaust, Th. Schlumprecht, E. Odell,  \emph{On Asymptotic structure, the Szlenk
index and UKK properties inBanach spaces}, Positivity \textbf{3} (1999), 173-199.

\bibitem{L} G. Lancien, \emph{On uniformly convex and uniformly Kadec-Klee renormings}, Serdica Math.
J. \textbf{21} (1995), 1-18.

\bibitem{LPR} G. Lancien, A. Prochazka, M. Raja, \emph{Szlenk indices of convex hulls}, J. Funct. Anal., \textbf{272} (2017), 498-521. 

\bibitem{MS}  P. Motakis and Th. Schlumprecht, \emph{A metric interpretation of reflexivity
for Banach spaces}, to appear in Duke Math. J., arXiv:1604.07271. 


\bibitem{NP} I. Namioka, R. R. Phelps, \emph{Banach spaces which are Asplund spaces}, Duke Math. J., \textbf{42} 
(1975) no. 4, 735-750.


\bibitem{R} M. Raja, \emph{On weak$^*$ uniformly Kadec-Klee renormings}, Bull. London Math. Soc., \textbf{42}
(2010), 221-228.


\end{thebibliography}
\end{document}